\documentclass[a4paper,11pt,leqno]{article}	
\usepackage{graphicx}
\usepackage{relsize}
\usepackage{hyperref}

\usepackage{comment}

\medskipamount

\usepackage{amsmath}
\usepackage{amsthm}
\usepackage{amssymb}
\usepackage{amscd}
\usepackage{enumerate}
\usepackage{pstricks}
\usepackage{pst-node}
\usepackage{esint}
\usepackage{mathtools}

\theoremstyle{plain}
\newtheorem{theorem}{Theorem}[section]
\newtheorem{corollary}[theorem]{Corollary}
\newtheorem{lemma}[theorem]{Lemma}
\newtheorem{proposition}[theorem]{Proposition}

\theoremstyle{definition}
\newtheorem{assumption}[theorem]{Assumption}
\newtheorem{definition}[theorem]{Definition}

\theoremstyle{remark}
\newtheorem{remark}{Remark}

\def\bs{}

\newcommand\R{\mathbb{R}}

\newcommand\N{\mathbb{N}}

\newcommand\Ra{R}

\renewcommand\boldsymbol{}

\newcommand\ya{y}
\newcommand\Aa{A}
\newcommand\ua{u}

\newcommand\calW{\mathcal{W}}

\newcommand{\SO}[1]{\operatorname{SO}(#1)}
\newcommand\id{I}
\newcommand\sym{\operatorname{sym}}
\newcommand\skw{\operatorname{skw}}

\newcommand\transpose{\textrm{t}}

\newcommand\meas{\operatorname{meas}}

\newcommand\esssup{\mathop{\operatorname{ess\,sup}}}
\newcommand\ud{\,\mathrm{d}}
 
 \newcommand\dist{\operatorname{dist}}
\newcommand{\e}{\varepsilon}
\newcommand{\eh}{{\varepsilon{\scriptstyle(h)}}}
\newcommand\eps{\varepsilon}
\newcommand\diver{\mathop{\operatorname{div}}}

\DeclareMathOperator{\inte}{int}

\newcommand{\step}[1]{\noindent \textbf{Step #1.}}

\newcommand\loc{\textrm{loc}}
\newcommand\ee{e}

\title{On the general homogenization of von K\'arm\'an plate equations from $3D$ nonlinear elasticity}
\date{\today}

\author{Igor Vel\v{c}i\'c\\
University of Zagreb, Faculty of Electrical Engineering and Computing,\\
  Unska 3, 10000 Zagreb, Croatia,\\
igor.velcic@fer.hr}

\begin{document}

\maketitle

\begin{abstract}
Starting from $3D$ elasticity equations we derive the model of the homogenized von K\'arm\'an plate by means of $\Gamma$-convergence. This generalizes the recent results, where the material oscillations were assumed to be periodic.
\vspace{10pt}


\end{abstract}

\tableofcontents[1]

\section{Introduction} \label{sectionprvi}
This paper is about derivation of homogenized von K\'arm\'an plate equations, starting from $3D$ elasticity by means of $\Gamma$-convergence. We do not presuppose any kind of periodicity, but work in a general framework.
There is  vast literature on deriving plate equations from $3D$ elasticity. For the approach using formal asymptotic expansion see \cite{Ciarletplates} and the references therein.
The first  work on deriving the plate models by means of $\Gamma$-convergence was \cite{LDRa95} where the authors derived the membrane plate model. It was well known that the obtained models depend on the assumption what is the relation of the external loads (i.e., the energy) with respect to the thickness of the body $h$. Higher ordered models (such as bending and von K\'arm\'an plate models) are also derived by means of $\Gamma$-convergence (see \cite{FJM-02,FJM-06}). The key mathematical ingredient in these cases was the theorem on geometric rigidity.

In \cite{Braides-Fonseca-Francfort-00} (see also \cite{Babadjian-Baia-06}) the influence of the different inhomogeneities in the combination with dimensional reduction on the limit model was analyzed. These models are obtained in the membrane regime. Recently, the techniques from \cite{FJM-02,FJM-06} were combined together with two-scale convergence to obtain the models of homogenized rod in the bending regime (see \cite{Neukamm-11}), homogenized von  K\'arm\'an plate (see \cite{Velcic-12,NeuVel-12}), homogenized von K\'arm\'an shell (see \cite{Hornungvel12}) and homogenized bending plate (see \cite{Horneuvel12,Vel13}). These models were derived under the assumption of periodic oscillations where it was assumed that the material oscillates on the scale $\eps(h)$, while the thickness of the body is $h$. The obtained models depend on the parameter $\gamma=\lim_{h \to 0} \tfrac{h}{\eh}$. In the case of von K\'arm\'an plate the situation $\gamma=0$ corresponds to the case when dimensional reduction dominates and the obtained model is the model of homogenized von K\'arm\'an plate and can be obtained as the limit case when $\gamma \to 0$. Analogously, the situation when $\gamma=\infty$ corresponds to the case when homogenization dominates and can again be obtained as the limit when $\gamma  \to \infty$; this is the model of von K\'arm\'an plate obtained starting from homogenized energy. In the case of von K\'arm\'an shell and bending plate the situation $\gamma=0$ was more subtle and leaded that the models depend on the further assumption of the relation between $\eps(h)$ and $h$. We obtained different models for the case $\eh^2 \ll h \ll \eh$ and $h \sim \eh^2$.

Here we analyze the case of the simultaneous homogenization and dimensional reduction in the von K\'arm\'an regime in the general framework, without any assumption on the periodicity. Moreover, we do not suppose that the oscillations of the material are only in the in-plane directions (as opposed to periodic oscillations in \cite{NeuVel-12}), but can happen in any direction (even cross sectional). We obtain kind of stability result for the equations, i.e.,  in the limit, we always obtain the equations of von K\'arm\'an type model.
Simultaneous homogenization and dimensional reduction, without any assumption on periodicity, were also considered in a non-variational framework (see \cite{Courilleau04} for monotone nonlinear elliptic systems and \cite{gustafsson06} for linear elasticity system). In these papers compensated compactness arguments were used and the notion of $H$-convergence (introduced by Murat and Tartar, see \cite{TarMur97}) was adapted to the dimensional reduction.

This paper is the first treatment of simultaneous homogenization and dimensional reduction without periodicity assumption by variational techniques in the context of higher order models in elasticity, at least to the author's knowledge (membrane case is already analyzed in \cite{Braides-Fonseca-Francfort-00}).
The case of bending rod model is also recently analyzed by using the approach developed here (see \cite{MarVel-14}), which is the generalization of the periodic case analyzed in \cite{Neukamm-11}.
Here we restrict ourselves to von K\'arm\'an plate model where the linearization is already dominated, although the system itself is nonlinear.

We prove the validity of the following asymptotic formulae for the energy density
\begin{equation*}
Q(x'_0,M_1,M_2)=\lim_{r \to 0} \tfrac{1}{|B(x'_0,r)|} K\left(M_1+x_3M_2,B(x'_0,r)\right), \forall M_1,M_2 \in \R^{2 \times 2}_{\sym} \text{ and a.e. } x'_0 \in \omega,
\end{equation*}
where
\begin{eqnarray*}
K\left(M_1+x_3M_2,B(x'_0,r)\right) &=&\inf \Big\{\liminf_{h \to 0}  \int_{B(x'_0,r) \times I} Q^{h}\left(x,\iota(M_1+x_3M_2)+\nabla_{h} \psi^{h}\right) \, dx: \\ \nonumber & &\hspace{5ex}  (\psi_1^{h},\psi_2^{h},h\psi_3^{h}) \to 0 \textrm{ strongly in } L^2\left(B(x'_0,r) \times I,\R^3\right) \Big\},
\end{eqnarray*}
$B(x'_0,r)$ is a ball of radius $r$ and center $x'_0$ in $\R^2$, $\omega \subset \R^2$ is a Lipschitz domain, representing the plate, $I=[-\tfrac{1}{2},\tfrac{1}{2}]$, $x=(x',x_3)$, $\iota$ is the natural injection from $\R^{2 \times 2}$ to $\R^{3 \times 3}$ (see below) and $(Q^h)_{h>0}$ are quadratic functionals of $h$ problem (see Section \ref{sectiondrugi}).  This  formulae is  actually valid on a subsequence and  unifies all three regimes obtained in \cite{NeuVel-12}.

 We emphasize the fact that in the construction of recovery sequence for $\Gamma$-limit we do not use any kind of additional regularity of minimizers (or their higher integrability, see \cite{SW94}), that is usually done, but only equi-integrability of the minimizing sequence.
Besides equi-integrability of the minimizing sequence (i.e., the possibility to replace the gradients and scaled gradients by equi-integrable ones) there are two more key points of the proofs.
 One is the characterization of the displacements that have bounded symmetric gradients on thin domains (we adapt the result proved in \cite[Theorem 2.3]{griso05}) and the other is the  characterization of the displacements that have the energy of order $h^4$ (see \cite[Proposition 3.1]{NeuVel-12}).
The proof of this proposition relied on the theorem on geometric rigidity and similar observations in \cite{FJM-06}. The new essential part was to correct the vertical displacements in order  to obtain a sequence which is bounded in $H^2$. This was not done in \cite{FJM-06}, since the authors did not need more information on the corrector to obtain the lower bound.

The main result of this paper are Theorem \ref{thm:1} ("lower bound") and Theorem \ref{thm:up1} ("upper bound", see Remark \ref{nappprep}). Together with Lemma \ref{podniz} (which claims that every sequence has a subsequence that satisfies Assumption \ref{ass:main}) they are analogue for standard $\Gamma$-compactness result. We show that, under von K\'arm\'an scaling assumption, we always obtain, on a subsequence, von K\'arm\'an type energy where the energy density is given by the above asymptotic formulae.

The novelty of this paper  consists in recognizing the above asymptotic formulae, but also in the
adaptation of the general $\Gamma$-convergence techniques to the case of  higher ordered  thin models in elasticity.
Although the basic approach is not difficult (see the end of Section \ref{sectiondrugi}, where the strategy of the proofs is explained), it has its own peculiarities.

This paper is organized as follows: in Section \ref{sectiondrugi} we give general framework and main result, we also explain the strategy of the proofs, in Section \ref{sectiondokazi} we give the proofs of the main statements given in Section \ref{sectiondrugi} and  in the Appendix we prove some auxiliary claims.

\subsection{Notation}
If $x\in \R^3$ by $x'$ we denote $x'=(x_1,x_2)$. By $\nabla'$ we denote the operator
$\nabla' u=(\partial_1 u, \partial_2 u)$.
By $\bar{\R}$ we denote $\R \cup \{-\infty,+\infty\}$.
By $B(x,r)$  we denote the ball of radius $r$ with the center $x$ in the quadratic norm.
If $A \subset \R^n$, by $|A|$ we denote the Lebesgue measure of $A$.
If $A$ and $B$ are subsets of $\R^n$, by $A \ll B$ we mean
that  the closure $\bar{A}$ is contained in the interior $\inte(B)$ of $B$.

$\iota$ denotes the natural injection of $\R^{2\times 2}$ into $\R^{3\times 3}$.
Denoting the standard basis of $\R^3$ by $(e_1, e_2, e_3)$ it is given by
\begin{equation*}
\iota(A):= \sum_{\alpha,\beta=1}^2 A_{\alpha\beta}(e_\alpha\otimes e_\beta).
\end{equation*}
For $a,b \in \R^3$ by $a\wedge b$ we denote the wedge product of the vectors $a$ and $b$. By $\R^{n \times n}_{\sym}$ we denote the space of symmetric matrices of order $n$ while by $\R^{n \times n}_{\skw}$ we denote the space of antisymmetric matrices of order n. When we want to emphasize the dependence of  a sequence on variable $h$ we put sometimes   superscript $h$, sometimes superscript $h_n$. In the second case we want to emphasize the importance of depending on some predefined sequence (see ,e.g., Theorem \ref{thm:1}, Theorem \ref{thm:up1}), while in the first case such dependence is not so important and claims are satisfied for any sequence or even every $h$ (see, e.g., Theorem \ref{grisotm}, Lemma \ref{igor1}).

\section{General framework and main results} \label{sectiondrugi}
\paragraph{\bf The three-dimensional model.}
Throughout the paper $\Omega^h:=\omega\times(hI)$ denotes the
reference configuration of a thin plate with mid-surface
$\omega\subset\R^2$ and (rescaled) cross-section
$I:=(-\tfrac{1}{2},\tfrac{1}{2})$. We suppose that $\omega$ is Lipschitz domain, i.e., open, bounded and connected set with Lipschitz boundary.
We denote by $\Gamma=\partial \omega \times I$.
For simplicity we assume that
$\omega$ is centered, that is
\begin{equation}\label{ass:centered}
\int_\omega
  {\left(
      \begin{array}{c}
        x_1\\ x_2
      \end{array}\right)}\ud x_1\ud x_2=0.
\end{equation}
\begin{definition}[nonlinear material law]\label{def:materialclass}
  Let $0<\alpha\leq\beta$ and $\rho>0$. The class
  $\calW(\alpha,\beta,\rho)$ consists of all measurable functions
  $W\,:\,\R^{3 \times 3}\to[0,+\infty]$ that satisfy the following properties:
  \begin{align}
    \tag{W1}\label{ass:frame-indifference}
    &W\text{ is frame indifferent, i.e.}\\
    &\notag\qquad W(\Ra\boldsymbol F)=W(\boldsymbol F)\quad\text{ for
      all $\boldsymbol F\in\R^{3\times 3}$, $\boldsymbol R\in\SO
      3$;}\\
    \tag{W2}\label{ass:non-degenerate}
    &W\text{ is non degenerate, i.e.}\\
    &\notag\qquad W(\boldsymbol F)\geq \alpha\dist^2(\boldsymbol F,\SO 3)\quad\text{ for all
      $\boldsymbol F\in\R^{3\times 3}$;}\\
    &\notag\qquad W(\boldsymbol F)\leq \beta\dist^2(\boldsymbol F,\SO 3)\quad\text{ for all
      $\boldsymbol F\in\R^{3\times 3}$ with $\dist^2(\boldsymbol F,\SO 3)\leq\rho$;}\\
    \tag{W3}\label{ass:stressfree}
    &W\text{ is minimal at $\id$, i.e.}\\
    &\notag\qquad W(\id)=0;\\
    \tag{W4}\label{ass:expansion}
    &W\text{ admits a quadratic expansion at $\id$, i.e.}\\
    &\notag\qquad W(\id+\boldsymbol G)=Q(\boldsymbol G)+o(|\boldsymbol G|^2)\qquad\text{for all }\boldsymbol G\in\R^{3\times 3},\\
    &\notag\text{where $Q\,:\,\R^{3\times 3} \to\R$ is a quadratic form.}
  \end{align}
\end{definition}
In the following definition we state our assumptions on the family $(W^h)_{h>0}$
\begin{definition}[admissible composite material]\label{def:composite}
  Let $0<\alpha\leq\beta$ and $\rho>0$. We say that a family $(W^h)_{h>0}$
  \begin{equation*}
    W^h:\Omega\times \R^{3\times 3} \to \R^+\cup\{+\infty\}
  \end{equation*}
  describes an admissible composite material of class $\mathcal W(\alpha,\beta,\rho)$ if
  \begin{enumerate}[(i)]
  \item  For each $h>0$, $W^h$ is almost surely equal to a Borel function on $\Omega\times\R^{3\times 3}$,
  \item $W^h(x,\cdot)\in\mathcal W(\alpha,\beta,\rho)$ for  every $h>0$ and almost every $x\in\Omega$.
  \item there exists a monotone function $r:\R^+\to\R^+\cup\{+\infty\}$, such that $r(\delta)\to 0$ as
  $\delta\to 0$ and
  \begin{equation}\label{eq:94}
    \forall \boldsymbol G\in\R^{3\times 3}\,:\,\sup_{h > 0} \esssup_{x \in \Omega} |W^h(x,\id+\boldsymbol G)-Q^h(x,\boldsymbol G)|\leq|\boldsymbol G|^2r(|\boldsymbol G|),
  \end{equation}
  where $Q^h(x,\cdot)$ is a quadratic form given in Definition \ref{def:materialclass}.
  \end{enumerate}
\end{definition}
Notice that $Q^h$ can be written as the pointwise limit
\begin{equation} \label{defQ} (x,G) \to Q^h(x,G) := \lim_{\e \to 0}
\tfrac{1}{\e^2} W^h(x, Id+\e G),
\end{equation}
and therefore inherits the measurability properties of $W^h$.

\begin{lemma}
  \label{lem:111}
  Let $(W^h)_{h>0}$ be as in Definition~\ref{def:composite} and let $(Q^h)_{h>0}$ be the quadratic form
  associated to $W^h$ through the expansion \eqref{ass:expansion}. Then
  \begin{enumerate}[(i)]
  \item[(Q1)] for all $h>0$ and almost all $x\in\Omega$ the map $Q^h(x,\cdot)$ is quadratic and
      satisfies
    \begin{equation*}
      \alpha|\sym \boldsymbol G|^2\leq Q^h (x,\boldsymbol G)=Q^h(x,\sym \boldsymbol G)\leq \beta|\sym \boldsymbol G|^2\qquad\text{for all $ \boldsymbol G\in\R^{3\times 3}$.}
    \end{equation*}
  \end{enumerate}
\end{lemma}
\begin{proof}
(Q1) follows from (W2).
\end{proof}
\begin{remark}\label{ocjenkv}
From (Q1) it follows
\begin{eqnarray} \label{ocjenakv1} && \\ \nonumber
|Q^h(x,G_1)-Q^h(x,G_2)| &\leq& \beta |\sym G_1-\sym G_2|\cdot |\sym G_1+\sym G_2|,\\ & & \nonumber \hspace{10ex} \forall h>0, G_1,G_2 \in \R^{3 \times 3}.
\end{eqnarray}
\end{remark}
In the von K\'arm\'an regime we look for the energy functionals
\begin{equation*}
  I^{h}(\ya):=
      \frac{1}{h^4}\int_\Omega W^h(x,\nabla_h\ya(x))\ud
    x,
\end{equation*}
imposing their finiteness.
Denote by $e_h(y)$
  \begin{equation}
  e_h(y)=\tfrac{1}{h^4} \int_\Omega \dist^2(\nabla_h y,\SO 3).
  \end{equation}
With von K\'arman model we associate the triple (the  limit deformation is rigid and the energy depends on the horizontal in-plane displacement $u$ and vertical displacement $v$, see Theorem \ref{thm:1})
\begin{equation*}
  (\bar{\bs R},\ua,v)\in \SO 3\times\mathcal A(\omega),\qquad
  \mathcal A(\omega):=\left\{\,  (\bs
  u,v)\,:\,\bs u\in H^1(\omega,\R^2),\,v\in H^2(\omega)\,\right\}.
\end{equation*}

The following definition and lemma can be found in \cite{NeuVel-12}. The definition is changed in the way that we require less, i.e.,  strong convergence in $L^2$ instead of weak convergence in $H^1$.
The role of the definition is to introduce the standard limit deformations in von K\'arm\'an regime.
We will give the proof of the uniqueness for the sake of completeness.
\begin{definition}
  \label{def:2}
  We say that a sequence  $(\ya^{h})_{h>0}\subset L^2(\Omega,\R^3)$ converges to a triple $(\bar{\bs R},\bs
  u,v)\in\SO3\times L^2(\omega,\R^2)\times L^2(\omega)$, and write  $\ya^{h}\to(\bar{\bs R},\bs u,v)$, if there exist rotations
  $(\bar{\bs R}^{h})_{h>0}$ and functions $(\bs u^{h})_{h>0} \subset  L^2(\omega,\R^2)$, $(v^{h})_{h>0} \subset  L^2(\omega)$ such that
  \begin{align} \label{eq:ig000}
    &(\bar{\bs R}^{h})^T\left(\int_I \ya^{h}(x',x_3)\,dx_3-\fint_\Omega\ya^{h}\,dx\right)=
    \left(\begin{array}{c}
        x'+h^2{\bs u}^{h}(x')\\
        h v^{h}(x')
      \end{array}\right),\\ \label{eq:ig00}
    &\ua^{h}\to\ua\text{ in }L^2(\omega,\R^2),\qquad v^{h}\to v\text{ in }L^2(\omega) \qquad\text{and}\qquad\bar\Ra^{h}\to\bar\Ra.
  \end{align}
\end{definition}

A limit in the sense of Definition~\ref{def:2} is not unique as it
stands. However, uniqueness is obtained modulo the following
equivalence relation on $L^2(\omega,\R^2)\times L^2(\omega)$:
\begin{multline*}
  (\bs u_1,v_1)\sim   (\bs u_2,v_2)\qquad:\Leftrightarrow\\
  \left\{\begin{aligned}
    &\bs u_2(x')=\bs u_1(x')+(\bs A-\tfrac{1}{2}\bs
    a\otimes\bs a)x'-v_1(x')\bs a\\
    &v_2(x')=v_1(\hat
    x)+\bs a\cdot x'
  \end{aligned}\right.\qquad\text{ for some }\bs a\in\R^2,\bs
  A\in \R^{2 \times 2}_{\skw}.
\end{multline*}

The proof of the following lemma is given in the Section \ref{sectiondokazi}.
\begin{lemma}[uniqueness]
  \label{L:1.2}
  Let $(\bar\Ra,\ua,v)$, $(\widetilde\Ra,\tilde\ua,\tilde
  v)\in \SO 3\times L^2(\omega,\R^2)\times L^2(\omega)$ and consider a
  sequence $(\ya^{h})_{h>0}$ that converges to $(\bar\Ra,\ua,v)$. Then
  \begin{equation*}
    \ya^{h}\to(\widetilde\Ra,\tilde\ua,\tilde
    v)\qquad\Leftrightarrow\qquad \widetilde\Ra=\bar\Ra\text{ and
    }(\ua,v)\sim(\tilde\ua,\tilde v).
  \end{equation*}
\end{lemma}
\subsection{Identification of limit equations}
The following two definitions are analogous for $\Gamma$-$\liminf$ and $\Gamma$-$\limsup$. They are used to obtain that the Assumption \ref{ass:main} is satisfied on a subsequence (see Lemma \ref{podniz}).
We define for the sequence $(h_n)_{n \in \N}$ which monotonly decreases to zero  and  arbitrary $A \subset \omega$ open and $M \in L^2(\Omega,\R^{2 \times 2}_{\sym})$,
\begin{eqnarray}
\label{gli} & &\\ \nonumber K^{-}_{(h_n)_{n \in \N}}(M,A)&=&\inf \Big\{\liminf_{n \to \infty}  \int_{A \times I} Q^{h_n}\left(x,\iota(M)+\nabla_{h_n} \psi^{h_n}\right) \, dx: \\ \nonumber & & \hspace{20ex} (\psi_1^{h_n},\psi_2^{h_n},h_n\psi_3^{h_n}) \to 0 \textrm{ strongly in } L^2(A \times I,\R^3) \Big\} \\ \nonumber
&=& \sup_{\mathcal{U} \subset \mathcal{N}(0)} \liminf_{n \to \infty}\inf_{\psi \in H^1(A \times I,\R^3) \atop (\psi_1,\psi_2,h_n\psi_3) \in \mathcal{U}}
 \int_{A \times I} Q^{h_n}\left(x,\iota(M)+\nabla_{h_n} \psi\right) \, dx, \\
\label{gls} & &\\ \nonumber K^{+}_{(h_n)_{n \in \N}}(M,A)&=&\inf \Big\{\limsup_{n \to \infty} \int_{A \times I} Q^{h_n}\left(x,\iota(M)+\nabla_{h_n} \psi^{h_n}\right) \, dx:\\ \nonumber & & \hspace{20ex} (\psi_1^{h_n},\psi_2^{h_n},h_n \psi_3^{h_n}) \to 0 \textrm{ strongly in } L^2(A \times I,\R^3) \Big\} \\ \nonumber
&=& \sup_{\mathcal{U} \subset \mathcal{N}(0)} \limsup_{n \to \infty} \inf_{\psi \in H^1(A \times I,\R^3) \atop (\psi_1,\psi_2,h_n \psi_3) \in  \mathcal{U}} \int_{A \times I} Q^{h_n}\left(x,\iota(M)+\nabla_{h_n} \psi\right) \, dx.
\end{eqnarray}
By $\mathcal{N}(0)$ we have denoted the family of all neighborhoods of zero in the strong $L^2$ topology.
\begin{remark}\label{minat1}
Since the above expressions are monotonly decreasing in  $\mathcal{N}(0)$ it is enough to take the supremum on the monotone sequence of neighborhoods that shrinks to $\{0\}$ e.g. the sequence of (open or closed) balls of radius $r$, when $r \to 0$.
\end{remark}
\begin{remark}\label{barbara1000}
Notice that instead of $h_n \psi^{h_n}_3$ we could introduce the variable $\tilde{\psi}^{h_n}_3$ and then we could look for $\Gamma$-limit in zero of the changed functional in strong $L^2$ topology. However the obtained functional is not coercive in gradient and thus we can not use standard abstract theory developed for these kind of functionals (see \cite{DM93}).
\end{remark}
\begin{remark}\label{minat2}
By using standard diagonalization argument  it can be shown that for any $(h_n)_{n\in \N}$ monotonly decreasing to zero and any $A \subset \omega$ open and $M \in L^2(\Omega,\R^{2 \times 2}_{\sym})$ it holds
\begin{eqnarray*}
 \nonumber K^{-}_{(h_n)_{n \in \N}}(M,A)&=&\min \Big\{\liminf_{n \to \infty}  \int_{A \times I} Q^{h_n}\left(x,\iota(M)+\nabla_{h_n} \psi^{h_n}\right) \, dx: \\ \nonumber & & \hspace{15ex} (\psi_1^{h_n},\psi_2^{h_n},h_n\psi_3^{h_n}) \to 0 \textrm{ strongly in } L^2(A \times I,\R^3) \Big\}, \\ \nonumber K^{+}_{(h_n)_{n \in \N}}(M,A)&=&\min \Big\{\limsup_{n \to \infty} \int_{A \times I} Q^{h_n}\left(x,\iota(M)+\nabla_{h_n} \psi^{h_n}\right) \, dx:\\ \nonumber & & \hspace{15ex} (\psi_1^{h_n},\psi_2^{h_n},h_n \psi_3^{h_n}) \to 0 \textrm{ strongly in } L^2(A \times I,\R^3) \Big\}.
\end{eqnarray*}

\end{remark}

Let $\mathcal{D}$ denote the countable family of open subsets of $\omega$ which is dense (see Definition \ref{defap00}) and such that every $D \in \mathcal{D}$ is of class $C^{1,1}$.
The following lemma uses the standard diagonalization argument and Lemma \ref{lem:ocjena}. Since the proof is easy we will give it immediately here, although we are in this way violating a bit the structure of the paper.
\begin{lemma} \label{podniz}
For every sequence $(h_n)_{n \in \N}$ monotonly decreasing to zero there exists a subsequence, still denoted by $(h_n)_{n \in \N}$, such that
$$ K^+_{(h_n)_{n \in \N}}(M, D)=K^{-}_{(h_n)_{n \in \N}} (M,D),\quad \forall M \in L^2(\Omega,\R^{2 \times 2}_{\sym}), \ \forall D \in \mathcal{D}.$$
\end{lemma}
\begin{proof}
Take a countable family $\{M_j\}_{j \in N}\subset  L^2 (\Omega,\R^{2 \times 2}_{\sym})$
which is dense in $L^2 (\Omega,\R^{2 \times 2}_{\sym})$.
By a diagonalizing argument it is not difficult to construct a subsequence $(h_n)_{n \in \N}$ monotonly decreasing to zero such that
$$K^+_{(h_n)_{n \in \N}}(M_j, D)=K^{-}_{(h_n)_{n \in \N}} (M_j,D),\quad \forall j \in \N, \ \forall D \in \mathcal{D}.$$
From Lemma \ref{lem:ocjena} and density we have the claim.
\end{proof}
We have proved that for fixed sequence the following assumption is always valid on a subsequence.
\begin{assumption}
  \label{ass:main}
   For a sequence $(h_n)_{n \in \N}$ monotonly decreasing to zero we suppose that for every $D \in \mathcal{D}$  and every $M \in  L^2 (\Omega,\R^{2 \times 2}_{\sym})$  there exists $K(M,D)$ such that we have
  $$K^{+}_{(h_n)_{n \in \N}}(M,D)=K^{-}_{(h_n)_{n \in \N}}(M,D)=:K(M,D). $$
\end{assumption}
\begin{remark} \label{stefan100}
The Assumption \ref{ass:main} is the assumption needed for Theorem \ref{thm:1} and Theorem \ref{thm:up1} to be valid. Lemma \ref{podniz} states that the Assumption (and thus the mentioned theorems) are satisfied on a subsequence.
Theorem \ref{thm:1} and Theorem \ref{thm:up1} are analogous for standard $\Gamma$-compactness result (see \cite{Braidesgamma02,DM93}).
\end{remark}
The following lemma is easy to prove by a contradiction (see the proof of \cite[Proposition 1.44]{Braidesgamma02}).
\begin{lemma}\label{podniz1}
  Suppose that for a given sequence $(h_n)_{n \in \N}$ monotonly decreasing to zero the following is valid:
  for every $M \in L^2(\Omega,\R^{2 \times 2}_{\sym})$ and every $D \in \mathcal{D}$ there exists $K(M,D)$ such that every subsequence  $(h_{n(k)})_{k \in \N}$ monotonly decreasing to zero has a subsequence, still denoted by $(h_{n(k)})_{k \in \N}$, which satisfies
$$K(M,D)=K^{-}_{(h_{n(k)})_{k \in \N}}(M,D).   $$
Then the Assumption \ref{ass:main} is satisfied.
\end{lemma}

We introduce the space $\mathcal{S}_{vK}(\omega) \subset L^2(\Omega,\R^{2 \times 2}_{\sym})$ of matrix fields which appear as limit strains in von K\'arm\'an model
$$\mathcal{S}_{vK}(\omega)=\{M_1+x_3 M_2:\, M_1, M_2 \in L^2(\omega,\R^{2 \times 2}_{\sym}) \}. $$
\begin{remark} \label{razjasnjenje}
Starting from Lemma \ref{podniz} and Assumption \ref{ass:main} we could state the results, only restricting ourselves on the space $\mathcal{S}_{vK}(\omega)$ instead of $L^2(\Omega,\R^{2 \times 2}_{\sym})$. We refrained ourselves from doing so for the sake of generality, when it was meaningful.
\end{remark}

The proofs of the following claims are put in the Section \ref{sectiondokazi}. The next proposition is one of the key claims of the paper.
It corresponds to standard claim of integral representation of $\Gamma$-limit. Here we are able to specify that $Q$ is quadratic with some additional properties.
\begin{proposition}\label{identi}
Let  $(h_n)_{n\in \N}$ be a sequence for which the Assumption~\ref{ass:main} is satisfied.
Then there exists a function $Q:\omega \times \R^{2\times 2}_{\sym}\times \R^{2\times 2}_{\sym} \to \R$ (dependent on this sequence) such that for every
$A \subset \omega$ open and every $M \in \mathcal{S}_{vK}(\omega)$ of the form
$$ M=M_1+x_3 M_2, \text{ for some } M_1, \ M_2 \in L^2(\omega,\R^{2 \times 2}_{\sym}),$$
 we have
\begin{equation} \label{integralineq}
 K(M,A)=\int_{A} Q(x',M_1(x'), M_2(x')) \, dx'.
\end{equation}
Moreover $Q$ satisfies the following property
  \item[(Q'1)] for  almost all $x'\in\omega$ the map $Q(x',\cdot,\cdot)$ is a quadratic form and
      satisfies
    \begin{equation*}
      \tfrac{\alpha}{12}\left(|\boldsymbol G_1|^2+|\boldsymbol G_2|^2\right)\leq Q (x',\boldsymbol G_1,G_2)\leq \beta\left(| \boldsymbol G_1|^2+|G_2|^2\right)\qquad\text{for all $ \boldsymbol G_1,\ G_2\in\R^{2 \times 2}_{\sym}$.}
    \end{equation*}
\end{proposition}

The following lemma tells us what can be alternative to the Assumption \ref{ass:main}.
\begin{lemma} \label{lem:pojasnjenje}
Let $(h_n)_{n \in \N}$ be a sequence monotonly decreasing to zero.
Assume that for almost every $x' \in \omega$ there exists a sequence $(r_m^{x'})_{m\in \N}$ converging to zero
and numbers $K\left(M,B(x',r_m^{x'})\right)$ such that
\begin{eqnarray*}
&&K\left(M,B(x',r_m^{x'})\right):=K^{-}_{(h_n)_{n \in \N}}\left(M,B(x',r_m^{x'})\right)=K^{+}_{(h_n)_{n \in \N}}\left(M,B(x',r_m^{x'})\right), \\ && \hspace{10ex}\ \forall m \in \N, M=M_1+x_3M_2; M_1,M_2 \in \R^{2 \times 2}_{\sym}.
\end{eqnarray*}
Then for every $M \in \mathcal{S}_{vK}(\omega)$, $A \subset \omega$ open there exists $K(M,A)$ such that  the following property holds
$$K(M,A)=K^{-}_{(h_n)_{n \in \N}}(M,A)=K^{+}_{(h_n)_{n \in \N}}(M,A).$$
\end{lemma}

Denote by $I^0:\mathcal{A}(\omega) \to \R^+_0$ the functional
$$ I^0(u,v)=\int_\omega Q(x', \sym \nabla u+\tfrac{1}{2} \nabla v \otimes \nabla v,-\nabla^2 v)\, dx'. $$
Notice that for $(\ua_1,v_1),(\ua_2,v_2)\in\mathcal A(\omega)$ with $(\ua_1,v_1)\sim(\ua_2,v_2)$ we have $I^0(\ua_1,v_1)=I^0(\ua_2,v_2)$.

The following two theorems are  the main results of the paper.

\begin{theorem}
  \label{thm:1}
  Let  $(h_n)_{n\in \N}$ be a sequence for which the Assumption~\ref{ass:main} is satisfied.
  \begin{enumerate}[(i)]
  \item (Compactness).  Let $(\ya^{h_n})_{n\in\N}\subset H^1(\Omega,\R^3)$ be a sequence
    with equibounded energy, that is $\limsup\limits_{n\to
          \infty}I^{h_n}(\ya^{h_n})<\infty$.
    Then there exists $(\bar\Ra,\ua,v)\in\SO 3\times\mathcal
    A(\omega)$ such that $\ya^{h_n}\to (\bar\Ra,\ua,v)$ up to a subsequence.

  \item (Lower bound). Let $(\ya^{h_n})_{n \in\N}\subset H^1(\Omega,\R^3)$ be a sequence
    with equibounded energy. Assume that $\ya^{h_n}\to(\bar\Ra,\ua,v)$. Then
    \begin{equation*}
      \liminf\limits_{n\to
        \infty}I^{h_n}(\ya^{h_n})\geq I^0(\ua,v).
    \end{equation*}
  \end{enumerate}
\end{theorem}

\begin{theorem}
  \label{thm:up1}
  Let $(h_n)_{n\in \N}$ be a sequence monotonly decreasing to zero for which the Assumption~\ref{ass:main} is satisfied.
  Then for every $(\bar\Ra,\ua,v)\in\SO
    3\times\mathcal A(\omega)$ there exists a subsequence, still denoted by $(h_n)_{n \in \N}$, such that
   $(\ya^{h_n})_{n \in \N} \subset
    H^1(\Omega,\R^3)$ with
    \begin{equation*}
      \ya^{h_n}\to(\bar\Ra,\ua,v)\qquad\text{and}\qquad
      \lim\limits_{n \to
        \infty}I^{h_n}(\ya^{h_n})=I^0(\ua,v).
    \end{equation*}
\end{theorem}
\begin{remark} \label{nappprep}
Although  Theorem \ref{thm:up1} is stated in the form such that subsequence can depend on $(\bar{R},u,v)$, by separability and diagonal procedure we conclude that there exists a subsequence such that the claim of the Theorem is valid for every $(\bar\Ra,\ua,v)\in\SO
    3\times\mathcal A(\omega)$. Nevertheless, it can be easily checked that Theorem \ref{thm:1} and Theorem \ref{thm:up1} do imply in their form the convergence of minimizers  to the minimizer of the limit functional (if we also add to them appropriate external loads, see \cite{FJM-06}), on every subsequence of the sequence $(h_n)_{n\in \N}$, where they are converging. Thus we can even the statement of  Theorem \ref{thm:up1} in its form look at as upper bound statement.
\end{remark}
\subsubsection{Strategy of the proofs} \label{stefan101}
The proofs use some standard approach of $\Gamma$-convergence adapted to this situation with specific compactness result. The essential part of the proof of Theorem \ref{thm:1} consists in writing the sequence of limiting strains given by
$$ E^{h}=\frac{\sqrt{(\nabla_{h} y^{h})^t \nabla_{h} y^{h}}-I}{{h}^2},$$
in the form (up to term which converges to zero strongly in $L^2$ and on a "large set")
 \begin{equation} \label{forma1}
 E^{h}\approx\underbrace{\sym \nabla u-\frac{1}{2} \nabla v \otimes \nabla v-x_3 \nabla^2 v}_{\textrm{limiting strain}}+\sym \nabla_{h} \psi^{h},
 \end{equation}
where $(\psi^{h}_1,\psi^{h}_2,{h}\psi^{h}_3) \to 0$ strongly in $L^2$ and $(\sym \nabla_{h} \psi^{h})_{{h}>0}$ is bounded in $L^2$. In order to do that we use the result \cite[Proposition 3.1]{NeuVel-12} which identifies the behavior of deformations of energy of order $h^4$ (von K\'arm\'an regime). This identification enables us to define the limiting energy, where we relax the energy with respect to the field $\nabla_{h} \psi^{h}$.
This naturally imposes that in the asymptotic formulae we separate fixed $L^2$ field which represents the limit strain and variable relaxation field $\nabla_{h} \psi^{h}$ (since the density satisfies the property (Q1)  defining the relaxation field as $\nabla_{h} \psi^{h}$ or $\sym \nabla_{h} \psi^{h}$ is the same). That is how we come to the definition of the functional $K(\cdot,\cdot)$. This separation is not done in standard $\Gamma$-convergence techniques, but here it is also helpful for proving the properties of the functional $K(\cdot,\cdot)$, i.e., to establish its representation, given in Proposition \ref{identi}.

In order to make the lower bound we need to use the truncation argument, but we have to preserve the structure of symmetrized scaled gradients.
This is done by using a bit modified Griso's decomposition which  says that (see Proposition \ref{igor2})
\begin{equation} \label{forma2}
 \psi^{h}\approx\left(\begin{array}{c} 0\\ 0\\ \tfrac{\varphi^h}{h} \end{array} \right)-x_3 \left(\begin{array}{c} \partial_1 \varphi^h \\ \partial_2 \varphi^h\\ 0 \end{array} \right)+ \bar{\psi}^{h},
\end{equation}
i.e.,
\begin{equation} \label{forma2}
\sym \nabla_{h} \psi^{h}=-x_3\iota(\nabla'^2 \varphi^{h})+\sym \nabla_{h} \bar{\psi}^{h},
\end{equation}
up to a term converging to zero strongly in $L^2$. Here we have that $\varphi^{h} \to 0$ ($\varphi^{h}$ are only functions of in-plane variables) and $\bar{\psi}^{h} \to 0$ strongly in $L^2$ and $(\nabla^2 \varphi^{h})_{h>0}$ and $(\nabla_{h} \bar{\psi}^{h})_{h>0}$ are bounded in $L^2$.
(in fact, a slight modification of \cite[Proposition 3.1]{NeuVel-12}, that is done directly in the proof of lower bound,  gives us directly the strain in  this form, i.e., in expression (\ref{forma1}), $\sym \nabla_{h} \psi^{h}$  is replaced by (\ref{forma2})).
We can then replace both of these sequences by equi-integrable sequences recalling some standard results from the literature (see \cite{FoMuPe98,BoceaFon02,BraidesZeppieri07}). The adaptation of these standard results to this situation is analyzed in the Appendix.
 The proof of Proposition \ref{identi} follows, in its basic ideas, the standard approach for obtaining integral representation (and for proving quadraticity). Of course, we use the convenience that we are dealing with quadratic functionals. Instead of using De Giorgi slicing argument, we again use the equi-integrability  property of minimizing sequence (for this we now need Griso's decomposition).
 In the proof of the upper bound we use the asymptotic formulae and we replace the relaxation field with the field of scaled gradients bounded in $L^\infty$. This can be done by adaptation of Lusin's type approximation already exploited in the derivation of lower dimensional models (see \cite{FJM-02}).
 The basic fact is that we can replace the equi-integrable sequence by $L^\infty$ sequence, and that the estimate of the error is uniform with respect to the index of the sequence.
 To summarize: we prove lower and upper bound by using only equi-integrability of the minimizing sequence; Griso's decomposition is used to write asymptotic formulae in different way, i.e., to obtain expression (\ref{forma2}), while \cite[Proposition 3.1]{NeuVel-12} is used to decompose the sequence of strains in the form given by the expressions (\ref{forma1}) and (\ref{forma2}).

\section{Proofs} \label{sectiondokazi}
\subsection{Proof of Lemma \ref{L:1.2}}
\begin{proof}
Without loss of generality we can assume that $\bar{R}=I$.
Assume that $\ya^{h}\to (\id,\ua,v)$ and $\ya^{h}\to
(\widetilde\Ra,\widetilde\ua,\widetilde v)$. Then, by definition, there exist two sequences $(\bar
\Ra^{h},\ua^{h},v^{h})$ and $(\widetilde \Ra^{h},\tilde \ua^{h}, \tilde
v^{h})$ with
\begin{eqnarray} \label{eq:ig1}
& &\ua^{h}\to\ua,\ \widetilde u^{h} \to \widetilde \ua \text{ in
}L^2,\qquad v^{h}\to v,\ \widetilde v^{h} \to \widetilde v \text{ in }L^2,
\\ \nonumber& & \hspace{+20ex} \bar\Ra^{h}\to\id,\ \widetilde \Ra^{h} \to \widetilde \Ra,
\end{eqnarray}
as $n\to \infty$, and
\begin{eqnarray*}
  \fint_I\ya^{h}(x',x_3)\,dx_3=
  \bar\Ra^{h}
  \left(\begin{array}{c}
      x'+{h}^2\ua^{h}\\
      {h} v^{h}
    \end{array}\right)
  = \widetilde\Ra^{h}
  \left(\begin{array}{c}
      x'+{h}^2\widetilde\ua^{h}\\
      {h} \widetilde v^{h}
    \end{array}\right).
\end{eqnarray*}
Rearranging terms and introducing
$\hat\Ra^{h}:=(\widetilde\Ra^{h})^T\bar\Ra^{h}$ yields
\begin{eqnarray} \label{eq:ig2}
  & &\big(\hat\Ra^{h}-\id \big) \left( \begin{array}{c} x' \\ 0 \end{array} \right)+{h}\hat\Ra^{h} \left( \begin{array}{c}0 \\ v^{h}(x')  \end{array} \right)+
  {h}^2\hat\Ra^{h} \left( \begin{array}{c}\ua^{h}(x') \\ 0  \end{array} \right)\\ & &= {h}\left( \begin{array}{c} 0 \\ \nonumber \tilde v^{h} (x') \end{array} \right)+
  {h}^2 \left( \begin{array}{c}\tilde \ua^{h} (x')\\ 0  \end{array} \right),
\end{eqnarray}
for almost every $x'\in\omega$ and all $h$. In the limit $h \to 0$ we get $\big(\hat\Ra-\id\big)
\left( \begin{array}{c} x' \\ 0 \end{array} \right)=0$. Combined
with $\hat\Ra\in\SO 3$, and
$\hat\Ra=\widetilde\Ra^T\bar\Ra=\widetilde\Ra^T$,  this implies $\widetilde\Ra=\id$.
\smallskip

Set $\hat\Aa^{h}:=\frac{\hat\Ra^{h}-\id}{{h}}$. We claim that there exists
$\hat\Aa\in\R^{3 \times 3}_{\skw}$ such that
\begin{eqnarray}
  \label{eq:ig:st1}
   \hat\Aa^{h}&\to&\hat\Aa\qquad\text{with }\sym\hat\Aa=0,\\
  \label{eq:ig:st2}
  \frac{\sym\hat\Aa^{h}}{{h}}&\to&\frac{1}{2}\hat\Aa^2.
\end{eqnarray}
Here comes the argument. Dividing (\ref{eq:ig2}) by ${h}$, and
rearranging terms, yields
\begin{eqnarray}
\label{eq:igggg6}
& &\hat\Aa^{h} \left( \begin{array}{c} x' \\ 0 \end{array} \right)+  \left( \begin{array}{c}0 \\ v^{h}(x')  \end{array} \right)
+{h} \hat\Aa^{h} \left( \begin{array}{c}0 \\ v^{h}(x')  \end{array} \right)+ {h} \left( \begin{array}{c}\ua^{h}(x') \\ 0  \end{array} \right)+{h}^2 \hat\Aa^{h} \left( \begin{array}{c}\ua^{h}(x') \\ 0  \end{array} \right) \\ & &= \left( \begin{array}{c}0 \\ \nonumber \tilde v^{h} (x') \end{array} \right)+
{h} \left( \begin{array}{c}\tilde \ua^{h} (x')\\ 0  \end{array} \right).
\end{eqnarray}
 We deduce that the first term of the left hand side converges in $L^2(\omega)$.
 This implies that
 $\hat\Aa^{h}\ee_\alpha$ converges as $h\to 0$. From the identity $\hat
\Ra^{h}\ee_3=\hat \Ra^{h}\ee_1\wedge\hat \Ra^{h}\ee_2$, we deduce that
\begin{eqnarray*}
  \hat\Aa^{h} \ee_3=(\hat\Aa^{h}\ee_1\wedge\hat \Ra^{h}\ee_2+\ee_1\wedge\hat\Aa^{h}\ee_2),
\end{eqnarray*}
and thus $\hat\Aa^{h}$ converges to some limit $\hat\Aa\in\R^{3\times 3}$.
Eventually, the relation $(\hat\Aa^{h})^T \hat\Aa^{h}=-2\frac{\sym \hat\Aa^{h}}{{h}}$
yields \eqref{eq:ig:st1} and \eqref{eq:ig:st2}.
\smallskip

To complete the argument, it remains to prove that
\begin{eqnarray}\label{eq:ig6:st3a}
  \tilde v(x')&=&v(x')+\bs a\cdot x'\qquad\text{where }\bs
  a:=(\hat\Aa_{31},\hat\Aa_{32})\\
  \label{eq:ig6:st3}
   \tilde   \ua(x')&=& \ua (x')+( \bs A-\tfrac{1}{2}\bs
    a\otimes\bs a)x'-v(x')\bs a
\end{eqnarray}
for some skew symmetric matrix $\Aa\in\R^{2 \times 2}_{\skw}$. The first identity
appears in the limit $n\to \infty$ in the third component of identity
\eqref{eq:igggg6}. For the proof of  \eqref{eq:ig6:st3} we introduce the
skew-symmetric matrix $\Aa^h\in\R^{2\times 2}_{\skw}$
\begin{equation}
\label{eq:ig7} \Aa^{h}_{\alpha \beta}=\frac{\hat
\Aa_{\alpha\beta}^{h}}{{h}}- \frac{(\sym \hat \Aa^{h})_{\alpha
\beta}}{{h}}, \textrm{ for } \alpha,\beta=1,2.
\end{equation}
Going back to \eqref{eq:igggg6}, after dividing by ${h}$, we find that
${h}^{-1}\hat\Aa_{\alpha\beta}^h$, $\alpha,\beta\in\{1,2\}$, converges as ${h}\to 0$.
This implies $\hat A_{\alpha \beta}=0$.
Combined with
\eqref{eq:ig:st2} we deduce that $\Aa^{h}$ converges to some
$\Aa\in\R^{2\times 2}_{\skw}$. Now, a calculation yields \eqref{eq:ig6:st3} i.e. we divide \eqref{eq:igggg6} by ${h}$ and let $h \to 0$ in the first two components.

\step 2 Argument for ``$\Leftarrow$''.

Suppose that $\ya^{h} \in L^2(\omega;\R^3)$ converges to the
triple $(\bar \Ra,\ua,v)$ in the sense of definition
(\ref{def:2}). Let us now take arbitrary $\Aa \in \R^{2 \times 2}_{\skw}$ and $\bs a \in \R^2$, and set
\begin{equation} \label{eq:ig10}
\tilde \Ra^{h}=\bar \Ra^{h}\exp(-{h}^2 \iota(\Aa))\exp(-{h} \bs a_e),
\end{equation}
where $  a_e \in \R^{3\times 3}$ is defined by
\begin{equation} \label{eq:ig11}
 a_e:=\left(\begin{array}{cc} 0  &-\bs a  \\ \bs a^T & 0 \end{array} \right).
\end{equation}
We define $\tilde \ua^{h}$, $\tilde v^{h}$ via identity \eqref{eq:ig000}.
From the expansions
\begin{equation} \label{eq:ig12}
  \exp({h}^2 \iota(\Aa))=\id+{h}^2\iota(\Aa)+O({h}^4), \qquad \exp({h}\bs a_e)=\id+{h}\bs a_e+\frac{{h}^2}{2} \bs a_e^2+O({h}^3),
\end{equation}
we conclude that
\begin{eqnarray*}
 \tilde   \ua^{h}(x')&=& \ua^{h} (x')+( \bs A-\tfrac{1}{2}\bs
    a\otimes\bs a)x'-v^{h}(x')\bs a+O({h}),\\
    \tilde v^{h}(x')&=& v^{h} (\hat
    x)+\bs a\cdot x'+O({h}),
\end{eqnarray*}
where $\|O({h})\|_{L^2} \leq C{h}$, for some $C>0$.
\end{proof}

\subsection{Characterization of the symmetrized scaled gradients}

The following theorem is proved in \cite{griso05}. We use it to obtain the characterization of the sequence $\psi^h \in H^1(\Omega,\R^3)$ which satisfies the property that $(\sym \nabla_h \psi^h)_{h>0}$ is bounded in $L^2$ and $(\psi^h_1,\psi^h_2,h\psi^h_3) \to 0$, strongly in $L^2$ (see Proposition \ref{igor2}). In the claims below when we put $C(A)$ it means that the constant depends on the domain $A$, but not on functions and variable $h$.
\begin{theorem}\label{grisotm}
Let $A \subset \omega$ with Lipschitz boundary and $\psi \in H^1(A \times I,\R^3)$ and $h>0$. Then we have the following decomposition
$$ \psi(x)=\hat{\psi}(x')+r(x') \wedge x_3 e_3+\bar{\psi}(x)=\left\{\begin{array}{l} \hat{\psi}_1(x')+r_2(x')x_3+\bar{\psi}_1(x) \\
\hat{\psi}_2(x')-r_1(x')x_3+\bar{\psi}_2(x)\\
\hat{\psi}_3(x')+\bar{\psi}_3(x)
 \end{array}
 \right., $$
where
\begin{equation} \label{igor0}
 \hat{\psi}=\int_I \psi \ dx_3,\ r=\frac{3}{2} \int_I x_3 e_3 \wedge \psi(x) \ dx_3,
\end{equation}
and the following estimate is valid
\begin{equation}\label{eq:111}
\| \sym \nabla_h (\hat{\psi}+r\wedge x_3 e_3)\|^2_{L^2}+\|\nabla_h \bar{\psi}\|^2_{L^2}+\tfrac{1}{h^2}\|\bar \psi \|^2_{L^2} \leq C(A) \|\sym \nabla_h \psi\|^2_{L^2}.
\end{equation}
\end{theorem}
\begin{remark}
Notice that
\begin{eqnarray}
& & \label{teichmann} \| \sym \nabla_h (\hat{\psi}+r\wedge x_3 e_3)\|^2_{L^2(A \times I)} = \\
& & \nonumber
\|\sym \nabla' (\hat{\psi}_1, \hat{\psi}_2)\|^2_{L^2(A)}+\|\sym \nabla' (r_2,-r_1)\|^2_{L^2(A)}\\ \nonumber & &+\tfrac{1}{h^2}\|\partial_1 (h\hat{\psi}_3)+r_2\|^2_{L^2(A)}+\tfrac{1}{h^2}\|\partial_2 (h\hat{\psi}_3)-r_1\|^2_{L^2(A)}.
\end{eqnarray}
Thus from Korn's inequality it follows

\begin{eqnarray}\label{lukas1}
& &\|(\hat{\psi}_1,\hat{\psi}_2, h \hat{\psi}_3) \|^2_{H^1(A)}+\|(r_1,r_2)\|^2_{H^1(A)}+\tfrac{1}{h^2}\|\partial_1 (h\hat{\psi}_3)+r_2\|^2_{L^2(A)} \\ \nonumber & &+\tfrac{1}{h^2}\|\partial_2 (h\hat{\psi}_3)-r_1\|^2_{L^2(A)}\\ & & \nonumber  \leq C(A) \left(\| \sym \nabla_h (\hat{\psi}+r\wedge x_3 e_3)\|^2_{L^2(A \times I)}+\|r\|^2_{L^2(A)}+\|(\hat{\psi}_1,\hat{\psi}_2,h\hat{\psi}_3)\|^2_{L^2(A \times I)}\right)\\
\nonumber & & \leq C(A) \left(\| \sym \nabla_h (\hat{\psi}+r\wedge x_3 e_3)\|^2_{L^2(A \times I)}+\|(\psi_1,\psi_2,h\psi_3)\|^2_{L^2(A \times I)}\right).
\end{eqnarray}

\end{remark}

The following lemma is crucial for proving Proposition \ref{igor2}.
\begin{lemma} \label{igor1}
Let $A \subset \omega$ with $C^{1,1}$ boundary and $h>0$. If  $r \in H^1(A,\R^2)$  and $\hat{\psi} \in H^1(A,\R^3)$ is such that $\int_{A_i} \hat{\psi}_3 \, dx=0$, for every connected component $A_i$ of $A$ (see Lemma \ref{zzzadnje}), then  there exists $\varphi  \in  H^2(A)$ and
$w \in H^1(A)$ such that $\hat{\psi}_3=\frac{\varphi}{h}+w$ and
\begin{eqnarray}\label{jeldodano}
& &\|\varphi\|^2_{H^2(A)}+\|(\hat{\psi}_1,\hat{\psi}_2)\|^2_{H^1(A)}+\|r\|^2_{H^1(\omega)}+\|w\|^2_{H^1(A)}+\tfrac{1}{h^2}\|\partial_1 \varphi+r_2 \|^2_{L^2(A)}\\ \nonumber& &+\tfrac{1}{h^2}\|\partial_2 \varphi-r_1 \|^2_{L^2(A)} \leq \\ \nonumber & & \hspace{5ex}  C(A) \left(  \| \sym \nabla_h (\hat{\psi}+r\wedge x_3 e_3)\|^2_{L^2(A \times I)}+\|r\|^2_{L^2(A)}+\|(\hat{\psi}_1,\hat{\psi}_2,h\hat{\psi}_3) \|^2_{L^2(A \times I)}\right).
\end{eqnarray}
\end{lemma}
\begin{proof}
We do the regularization of $r$ in the similar way as in \cite[Proposition 3.1]{NeuVel-12} and \cite[Lemma 3.8]{Hornungvel12}.
We  look for the solution of the problem
\begin{equation} \label{elias1}
\min_{\varphi\in H^1(A) \atop \forall i, \ \int_{A_i} \varphi=0 }\int_A|\nabla' \varphi+(r_2,-r_1)|^2\,dx'.
\end{equation}
  The
  associated Euler-Lagrange equation of (\ref{elias1}) reads
  \begin{equation}\label{regular}
    \left\{\begin{aligned}
      -\triangle' \varphi=\,&\,\nabla'\cdot (r_2,-r_1)\qquad&&\text{in } A\\
      \partial_{\nu}\varphi=\,&\, -(r_2,-r_1)\cdot \nu \qquad&&\text{on }\partial A.
    \end{aligned}\right.
  \end{equation}
  Since $\nabla'\cdot (r_2,-r_1)\in L^2$, we obtain by
  standard  regularity estimates that  $\varphi\in H^2(A)$ and
  $\|\varphi\|_{H^2(A)}\leq C(A)  \|r\|_{H^1(A)}$, where we need  $C^{1,1}$ regularity of $\partial A$.
The claim follows from (\ref{teichmann}), (\ref{lukas1}) and the following inequalities
\begin{eqnarray}
\label{eq:prvaaa}\|\partial_1 \varphi+r_2 \|^2_{L^2(A)}  &\leq& \|\partial_1 (h\hat{\psi}_3)+r_2 \|^2_{L^2(A)}, \\
\label{eq:drugaaa}\|\partial_2 \varphi-r_1 \|^2_{L^2(A)}  &\leq& \|\partial_2 (h\hat{\psi}_3)-r_1 \|^2_{L^2(A)}, \\
\left\| \nabla' \left(\hat{\psi}_3-\tfrac{\varphi}{h}\right)\right\|^2_{L^2(A)} &\leq& \tfrac{1}{h^2}\|\partial_1 (h\hat{\psi}_3)+ r_2\|^2_{L^2}+\tfrac{1}{h^2}\|\partial_1 \varphi+ r_2\|^2_{L^2(A)} \\ & & \nonumber
+\tfrac{1}{h^2}\|\partial_2 (h\hat{\psi}_3)- r_1\|^2_{L^2(A)}+\tfrac{1}{h^2}\|\partial_2 \varphi- r_1\|^2_{L^2(A)}.
\end{eqnarray}
\end{proof}
\begin{remark} \label{periodic1}
 If we assume that $\hat{\psi}=0$, $r=0$ on $\partial A \times I$ (without the assumption that $\int_{A_i} \hat{\psi}_3 \, dx=0$, for every connected component $A_i$ of $A$)
 the claim of Lemma \ref{igor1} is valid and we can demand additionally that $\varphi \in H^2(A)$, $\varphi=0$ on $\partial A$, $w \in H_0^1(A)$. Using again the Korn's inequality with boundary conditions  we can omit
$\| r\|_{L^2}$ and $\|(\hat{\psi}_1,\hat{\psi}_2,h\hat{\psi}_3) \|_{L^2}$ on the right hand side of (\ref{jeldodano}).
To adapt the proof of Lemma \ref{igor1} instead of solving the problem (\ref{elias1}), we need to solve the problem
 \begin{equation}\label{regular2}
\min_{\varphi\in H_0^1(A)}\int_A|\nabla' \varphi+(r_2,-r_1)|^2\,dx'.
\end{equation}
\end{remark}
The following proposition gives the characterization of the symmetrized scaled gradients we will work with.

\begin{proposition}\label{igor2}
Let $A \subset \omega$ with $C^{1,1}$ boundary. Denote by $\{A_i\}_{i=1,\dots,k}$ the connected components of $A$.
\begin{enumerate}[(a)]
\item
Let $(\psi^h)_{h>0} \subset H^1(A \times I,\R^3)$ be such that
\begin{eqnarray}
\label{eq:pp1}&& (\psi_1^h,\psi_2^h,h\psi_3^h) \to 0, \text{ strongly in } L^2,\quad \forall h,i \int_{A_i} \psi_3^h=0, \\ && \limsup_{h\to 0} \|\sym \nabla_h \psi^h\|_{L^2(A)} \leq M<\infty.
 \end{eqnarray}
Then there exist $(\varphi^h)_{h>0} \subset H^2(A)$, $(\tilde{\psi}^h)_{h>0} \subset H^1(A \times I,\R^3)$
such that
$$ \sym \nabla_h \psi^h=-x_3\iota(\nabla'^2 \varphi^h)+\sym \nabla_h \tilde{\psi}^h+o^h,$$
where $o^h \in L^2(A \times I,\R^{3 \times 3})$ is such that $o^h \to 0$, strongly in $L^2$, and the following properties hold
\begin{eqnarray}
&&\label{igor10} \lim_{h \to 0} \left(\|\varphi^h\|_{H^1(A)}+\|\tilde{\psi}^h\|_{L^2(A\times I)} \right)=0,\\
&& \label{igor11} \limsup_{h \to 0} \left( \| \varphi^h \|_{H^2(A)}+\|\nabla_h \tilde{\psi}^h\|_{L^2(A\times I)} \right)\leq C(A) M.
\end{eqnarray}
\item
For every $(\varphi^h)_{h>0} \subset H^2(A)$, $(\tilde{\psi}^h)_{h>0} \subset H^1(A \times I,\R^3)$
such that
\begin{eqnarray*}
&&\lim_{h \to 0} \left(\|\varphi^h\|_{H^1(A)}+\|\tilde{\psi}^h\|_{L^2(A\times I)} \right)=0,\\
&& \limsup_{h \to 0} \left( \| \varphi^h \|_{H^2(A)}+\|\nabla_h \tilde{\psi}^h\|_{L^2(A)} \right)\leq M,
\end{eqnarray*}
there exists a sequence $(\psi^h)_{h>0} \subset H^1(A \times I,\R^3)$ such that
\begin{eqnarray*}
\sym \nabla_h \psi^h&=& -x_3 \iota(\nabla'^2 \varphi^h)+\sym \nabla_h \tilde{\psi}^h,\\
&& (\psi^h_1,\psi^h_2,h\psi^h_3) \to 0 \text{ strongly in }L^2,\\
&& \limsup_{h\to 0} \|\sym \nabla_h \psi^h\|_{L^2(A)} \leq 2 M.
\end{eqnarray*}
\end{enumerate}

\end{proposition}

\begin{proof}
The proof follows immediately from  Theorem \ref{grisotm} and Lemma \ref{igor1}. Namely, if we use the decomposition from Theorem \ref{grisotm} and Lemma \ref{igor1} we obtain
\begin{eqnarray} \label{igor5}
 \psi^h &=& \hat{\psi}^h+r^h\wedge x_3 e_3+\bar{\psi}^h \\ \nonumber &=&
\left(\begin{array}{c} \hat{\psi}_1^h \\ \hat{\psi}_2^h \\ 0  \end{array}\right)+\left(\begin{array}{c} 0\\ 0\\ \tfrac{\varphi^h}{h}+w^h \end{array} \right)-x_3 \left(\begin{array}{c} \partial_1 \varphi^h \\ \partial_2 \varphi^h\\ 0 \end{array} \right) +x_3 \left(\begin{array}{c} \partial_1 \varphi^h+r_2^h \\ \partial_2 \varphi^h-r_1^h\\ 0 \end{array} \right)+\bar{\psi}^h.
\end{eqnarray}

From the expressions (\ref{igor0}) and (\ref{eq:111}) it follows
$(\hat{\psi}_1,\hat{\psi}_2,h\hat{\psi}_3) \to 0$, $(r_1^h,r_2^h) \to 0$, $\bar{\psi}^h \to 0$ strongly in $L^2$. This implies $\varphi^h \to 0$ strongly in $L^2$.
Since it holds for every $i$, $\int_{A_i} \hat{\psi}_3 \, dx=0$ (see (\ref{igor0})), from (\ref{eq:111}) and Lemma \ref{igor1} we have that $\limsup_{h \to 0} \|w^h\|_{H^1(A)} \leq C(A)M$ and since $(\varphi^h)_{h>0}$  is bounded in $H^2$ we deduce by compactness that $\varphi^h\to 0$ strongly in $H^1$. Thus
we can find $(\tilde{w}^h)_{h>0} \subset H^2(A)$ such that
\begin{equation} \label{igor12}
 \lim_{h \to 0} \|w^h-\tilde{w}^h\|_{L^2(A)}=0,\ \limsup_{h \to 0} \|\tilde{w}^h\|_{H^1(A)}\leq C(A)M,\ \lim_{h \to 0} h \|\tilde{w}^h\|_{H^2(A)}=0.
\end{equation}
This can be done by mollifying $w^h$  with the mollifiers of  radius $\varepsilon^h \to 0$, $\varepsilon^h \gg h$.
From (\ref{igor5}) we have
\begin{eqnarray}
\psi^h&=&\left(\begin{array}{c} \hat{\psi}_1^h \\ \hat{\psi}_2^h \\ 0  \end{array}\right)+\left(\begin{array}{c} 0\\ 0\\ \tfrac{\varphi^h}{h}+\tilde{w}^h \end{array} \right)-x_3 \left(\begin{array}{c} \partial_1 \varphi^h \\ \partial_2 \varphi^h\\ 0 \end{array} \right) +x_3 \left(\begin{array}{c} \partial_1 \varphi^h+r_2^h \\ \partial_2 \varphi^h-r_1^h\\ 0 \end{array} \right)
\\ \nonumber & &+\left(\begin{array}{c} 0 \\ 0 \\ w^h-\tilde{w}^h \end{array} \right) +\bar{\psi}^h.
\end{eqnarray}
Now the claim follows from Lemma \ref{igor1}  and (\ref{igor12}) by defining
\begin{eqnarray*}
 \tilde{\psi}^h&=&\left(\begin{array}{c} \hat{\psi}_1^h \\ \hat{\psi}_2^h \\ 0  \end{array}\right)+x_3 \left(\begin{array}{c} \partial_1 \varphi^h+r_2^h \\ \partial_2 \varphi^h-r_1^h\\ 0 \end{array} \right)+hx_3 \left(\begin{array}{c} \partial_1 \tilde{w}^h \\ \partial_2 \tilde{w}^h\\ 0 \end{array} \right)+\left(\begin{array}{c} 0 \\ 0 \\ w^h-\tilde{w}^h \end{array} \right)+\bar{\psi}^h, \\
o^h &=& -hx_3\iota(\nabla'^2 \tilde{w}^h),
\end{eqnarray*}
after using the identity
$$ \sym \nabla_h \left(\begin{array}{c} 0 \\ 0 \\ \tilde{w}^h \end{array} \right)= \sym \nabla_h \left(\begin{array}{c} hx_3 \partial_1 \tilde{w}^h \\ hx_3 \partial_2 \tilde{w}^h \\ 0 \end{array} \right)-hx_3 \iota(\nabla'^2 \tilde{w}).$$
The second part of proposition is direct by defining
$$ \psi^h= \tilde{\psi}^h +\left(\begin{array}{c} 0\\ 0\\ \tfrac{\varphi^h}{h} \end{array} \right)-x_3 \left(\begin{array}{c} \partial_1 \varphi^h \\ \partial_2 \varphi^h\\ 0 \end{array} \right).   $$
\end{proof}

\subsection{Properties of $K(\cdot,\cdot)$}

Here we want to establish some important properties of $K(\cdot,\cdot)$ (see Lemma \ref{lem:svojstva}) which will help us to prove Proposition \ref{identi}. We emphasize the fact that all these properties are simple consequence of  Lemma \ref{lem:ocjena} and Lemma \ref{lem:glavnazamjena}.
\begin{lemma}[continuity in $M$] \label{lem:ocjena}
There exists a constant $C>0$ dependent only on $\alpha,\beta$ such that for each sequence $(h_n)_{n \in \N}$ monotonly decreasing to zero and $A \subset \omega$ open set is valid
\begin{eqnarray}\label{ocjena1111}
\left| K^{-}_{(h_n)_{n \in \N}}(M_1,A)-K^{-}_{(h_n)_{n \in \N}}(M_2,A)\right|&\leq& C \|M_1-M_2\|_{L^2}\left(\|M_1\|_{L^2}+\|M_2\|_{L^2}\right),\\\nonumber & &  \ \forall M_1,M_2 \in L^2(\Omega,\R^{2\times 2}_{\sym}),
\end{eqnarray}
and the analogous claim for $K^{+}_{(h_n)_{n \in \N}}$.
\end{lemma}
\begin{proof}
For fixed $M_1,M_2 \in L^2(\Omega,\R^{2\times 2}_{\sym})$ take an arbitrary $r,h>0$ and $\psi_\alpha^{r,h_n}\in H^1(A \times I,\R^3)$ that satisfies for $\alpha=1,2$
\begin{eqnarray}\label{definf}
& &  \int_\Omega Q^{h_n}\left(x,\iota(M_\alpha)+\nabla_{h_n} \psi_{\alpha}^{r,h_n}\right)\, dx \leq \\ \nonumber & & \hspace{10ex}\inf_{\psi \in H^1(A \times I,\R^3) \atop \|(\psi_1,\psi_2,h_n\psi_3)\|_{L^2} \leq r}
\int_{A \times I} Q^{h_n}\left(x,\iota(M_\alpha)+\nabla_{h_n} \psi\right) \, dx+h_n, \\
 \nonumber& & \| (\psi_{\alpha,1}^{r,h_n},\psi_{\alpha,2}^{r,h_n},h_n\psi_{\alpha,3}^{r,h_n})\|_{L^2} \leq r.
\end{eqnarray}
We want to prove that for every $r>0$ we have
\begin{eqnarray}\label{nped}
& &\left| \int_{A \times I} Q^{h_n}\left(x,\iota(M_1)+\nabla_{h_n} \psi_{1}^{r,h_n}\right)\, dx-\int_{A \times I} Q^{h_n}\left(x,\iota(M_2)+\nabla_{h_n} \psi_{2}^{r,h_n}\right)\, dx\right| \\ & & \nonumber \hspace{+15ex}\leq C \|M_1-M_2\|_{L^2}\left(\|M_1\|_{L^2}+\|M_2\|_{L^2}\right)+h_n.
\end{eqnarray}
From that we can  easily obtain (\ref{ocjena1111}).

Let us prove (\ref{nped}). From (\ref{definf}) and (Q1), by testing with zero function, we can assume for $\alpha=1,2$
$$
\alpha \|M_\alpha+ \sym \nabla_{h_n} \psi^{r,h_n}_{\alpha}\|^2_{L^2}  \leq \int_{A \times I} Q^{h_n}\left(x,\iota(M_\alpha)+\nabla_{h_n} \psi_{\alpha}^{r,h_n}\right)\, dx\leq \beta \|M_{\alpha}\|^2_{L^2}.
$$
From this we have for $\alpha=1,2$
\begin{equation} \label{identitet} \| \sym \nabla_{h_n} \psi^{r,h_n}_{\alpha}\|^2_{L^2} \leq C(\alpha,\beta) \|M_{\alpha}\|^2_{L^2}.
\end{equation}

Without any loss of generality we can also assume that
\begin{equation}\label{pretpostavka1}
\int_{A \times I} Q^{h_n} \left(x,\iota(M_1)+\nabla_{h_n} \psi_{1}^{r,h_n}\right)\, dx \geq \int_{A \times I} Q^{h_n}\left(x,\iota(M_2)+\nabla_{h_n} \psi_{2}^{r,h}\right)\, dx.
\end{equation}
We have
\begin{eqnarray*}
& &\left| \int_{A \times I} Q^{h_n}\left(x,\iota(M_1)+\nabla_{h_n} \psi_{1}^{r,h_n}\right)\, dx-\int_{A \times I} Q^{h_n}\left(x,\iota(M_2)+\nabla_{h_n} \psi_{2}^{r,h_n}\right)\, dx\right| \\ & & = \int_{A \times I} Q^{h_n}\left(x,\iota(M_1)+\nabla_{h_n} \psi_{1}^{r,h_n}\right)\, dx-\int_{A \times I} Q^{h_n}\left(x,\iota(M_2)+\nabla_{h_n} \psi_{2}^{r,h_n}\right)\, dx \\ & & =
 \int_{A \times I} Q^{h_n}\left(x,\iota(M_1)+\nabla_{h_n} \psi_{1}^{r,{h_n}}\right)\, dx-\int_{A \times I} Q^{h_n}\left(x,\iota(M_1)+\nabla_{h_n} \psi_{2}^{r,h_n}\right)\, dx \\ & & +\int_{A \times I} Q^{h_n}\left(x,\iota(M_1)+\nabla_{h_n} \psi_{2}^{r,h_n}\right)\, dx-
\int_{A \times I} Q^h\left(x,\iota(M_2)+\nabla_{h_n} \psi_{2}^{r,h_n}\right)\, dx \\
& & \leq h_n +C(\alpha,\beta) \|M_1-M_2\|_{L^2}\left(\|M_1\|_{L^2}+\|M_2\|_{L^2}\right).
\end{eqnarray*}
\end{proof}
The following lemma will be working lemma for establishing the properties of $K(\cdot,\cdot)$.
\begin{lemma}\label{lem:glavnazamjena}
Let for $(h_n)_{n \in \N}$ monotonly decreasing to zero Assumption \ref{ass:main} is satisfied.
Suppose that for $M \in  L^2(\Omega,\R^{2 \times 2}_{\sym})$ and $D \subset \omega$ with $C^{1,1}$ boundary  we have
$$K(M,D)=\lim_{n\to \infty} \int_{D\times I} Q^{h_n}(x,\iota(M)+\nabla_{h_n} \psi^{h_n})\, dx, $$
for some $(\psi^{h_n})_{n\in \N}$ such that $(\psi_1^{h_n},\psi_2^{h_n},h_n \psi_3^{h_n}) \to 0$ strongly in $L^2$. Then there exists a subsequence $(h_{n(k)})_{k\in \N}$ and $(\vartheta_k)_{k \in \N}\subset H^1(D \times I,\R^3)$ such that
\begin{enumerate}[(a)]
\item
$(\vartheta_{k,1},\vartheta_{k,2},h_{n(k)} \vartheta_{k,3}) \to 0$ strongly in $L^2$,
\item
 $(|\sym \nabla_{h_{n(k)}} \vartheta_k|^2 )_{k \in \N}$ is equi-integrable,
 $$ \sym \nabla_{h_{n(k)}} \vartheta_k=-x_3\iota( \nabla'^2 \varphi_{k})+\sym \nabla_{h_{n(k)}} \tilde{\psi}_{k},$$
 where $\left(|\nabla'^2 \varphi_k|^2\right)_{k\in\N}$ and $\left(| \nabla_{h_{n(k)}}\tilde{\psi}_k |^2\right)_{k \in \N}$ are equi-integrable and
 $\varphi_k \to 0$ strongly in $H^1$ and $\tilde{\psi}_k \to 0$ strongly in $L^2$.
 Also the following is valid
 $$\limsup_{k \to \infty} \left( \| \varphi_k \|_{H^2(D)}+\|\nabla_{h_{n(k)}}{\tilde{\psi}}_k\|_{L^2(D)}\right)\leq C(D) \left( \beta \|M\|^2_{L^2}+1 \right).$$
\item $$K(M,D)=\lim_{k\to \infty} \int_{D \times I} Q^{h_{n(k)}}(x,\iota(M)+\nabla_{h_{n(k)}} \vartheta_k)\, dx.$$
Moreover, one can additionally assume that for each $k \in \N$ we have $\vartheta_k=0$ on $\partial D \times I$, i.e., $\varphi_k=\nabla' \varphi_k=0$, on $\partial D$ and $\tilde{\psi}_k=0$ on $\partial D \times I$.
\end{enumerate}
\end{lemma}

\begin{proof}
We can without loss of generality assume that $\int_{D_i} \psi_3^{h_n} \, dx=0$, $\forall n \in \N$ and every connected component $D_i$ of $D$, $i=1,\dots,m$. By comparing with zero sequence one can additionally assume that
$$ \|\sym \nabla_{h_n} \psi^{h_n}\|_{L^2(D)} \leq \beta \|M\|^2_{L^2}+1, \quad \forall n \in \N.    $$
From Proposition \ref{igor2} we have that there exist $(\tilde{\varphi}_n)_{n \in \N} \subset H^2(D)$ and
$(\tilde{\tilde{\psi}}_n)_{n \in \N} \subset H^1(D\times I)$ such that
 \begin{equation} \label{jednakost1} \sym \nabla_{h_n} \psi^{h_n}=-x_3\iota(\nabla'^2 \tilde{\varphi}^{h_n})+\sym \nabla_{h_n} \tilde{\tilde{\psi}}^{h_n}+o^{h_n},
\end{equation}
where $o^{h_n} \in L^2(D \times I,\R^{3 \times 3})$ and the following properties hold
\begin{eqnarray}
&& \label{igor10aa} \lim_{n \to\infty} \left(  \| \tilde{\varphi}^{h_n} \|_{H^1(D)}+\| \tilde{\tilde{\psi}}^{h_n} \|_{L^2(D\times I)}+\|o^{h_n}\|_{L^2(D)} \right)=0,\\
&& \label{igor11bb} \limsup_{n \to \infty} \left( \| \tilde{\varphi}^{h_n} \|_{H^2(D)}+\|\nabla_{h_n} \tilde{\tilde{\psi}}^{h_n}\|_{L^2(D\times I)} \right)\leq C(D) \left( \beta \|M\|^2_{L^2}+1 \right).
\end{eqnarray}
Now we use Proposition \ref{ekvi1} and Theorem \ref{ekvi2} to obtain the sequence $(\varphi_k)_{k \in \N} \subset H^2(D)$ and $(\tilde{\psi}_k)_{k \in \N}$ such that
$\left(|\nabla'^2 \varphi_k|^2\right)_{k\in\N}$ and $\left(| \nabla_{h_{n(k)}} \tilde{\psi_k} |^2\right)_{k \in \N}$ are equi-integrable and
for $A_k$ defined by
\begin{eqnarray*}
A_k&:=&\{\varphi_k \neq \tilde{\varphi}^{h_{n(k)}}  \text{ or } \tilde{\tilde{\psi}}^{h_{n(k)}} \neq \tilde{\psi}_k \},
\end{eqnarray*}
is valid $|A_k| \to 0$ as $k \to \infty$ and
$$ \lim_{k \to\infty} \left(  \| \varphi_k \|_{H^1(D)}+\|\tilde{\psi}_k \|_{L^2(D\times I)} \right)=0. $$
Define $\vartheta_k$ as in Proposition \ref{igor2} by
$$\vartheta_k:= \tilde{\psi}_{k} +\left(\begin{array}{c} 0\\ 0\\ \tfrac{\varphi_k}{h_{n(k)}} \end{array} \right)-x_3 \left(\begin{array}{c} \partial_1 \varphi_k \\ \partial_2 \varphi_k \\ 0 \end{array} \right).$$
From the property (\ref{jednakost1}), Remark \ref{rem:jednakost}, equi-integrability and the fact that $|A_k| \to 0$ as $k \to \infty$ we have the following
\begin{eqnarray*}
K(M,D) &=& \lim_{n\to \infty} \int_{D \times I} Q^{h_n}(x,\iota(M)+\nabla_{h_n} \psi^{h_n})\, dx \\
&=& \lim_{n\to \infty} \int_{D \times I} Q^{h_n}(x,\iota(M)-x_3\iota( \nabla'^2 \tilde{\varphi}^{h_n})+\sym \nabla_{h_n} \tilde{\tilde{\psi}}^{h_n})\, dx \\
&\geq & \lim_{k\to \infty} \int_{A_k \times I} Q^{h_{n(k)}}\left(x,\iota(M)-x_3 \iota( \nabla'^2 \varphi_k)+\sym \nabla_{h_{n(k)}} \tilde{\psi}_k\right)\, dx \\
&=& \lim_{k\to \infty} \int_{D \times I} Q^{h_{n(k)}}(x,\iota(M) +\nabla_{h_{n(k)}} \vartheta_k )\, dx \geq K(M,D).
\end{eqnarray*}
The last inequality follows from the definition of $K(M,D)$ and the fact that $(\vartheta_{k,1},\vartheta_{k,2},$ $h_{n(k)} \vartheta_{k,3})$ $\to 0$ strongly in $L^2$.
 The last claim in (b) follows from the equi-integrability property and (\ref{igor11bb}).
The last claim in (c) follows from simple truncation Lemma \ref{nulizacija} below.
\end{proof}

\begin{lemma} \label{nulizacija}
Let $A\subset \omega$ be an open, bounded set.
Let $(\vartheta_n)_{n \in \N} \subset H^1(A \times I,\R^3)$ be defined by
$$\vartheta_n:= \psi_{n}+\left(\begin{array}{c} 0\\ 0\\ \tfrac{\varphi_n}{h_n} \end{array} \right)-x_3 \left(\begin{array}{c} \partial_1 \varphi_n \\ \partial_2 \varphi_n \\ 0 \end{array} \right),$$
where $(\psi_{n})_{n \in \N} \subset H^1(A \times I,\R^3)$ and $(\varphi_n)_{n \in \N} \subset H^2(A)$. Suppose that
$\left(|\nabla'^2 \varphi_n|^2\right)_{n\in\N}$ and $\left(| \nabla_{h_n} \psi_n |^2\right)_{n \in \N}$ are equi-integrable and
\begin{equation}
\lim_{n \to\infty} \left(  \| \varphi_n \|_{H^1(A)}+\| \psi_n \|_{L^2(A \times I)} \right)=0.
\end{equation}
Then there exist sequences $(\tilde{\varphi}_n)_{n \in \N} \subset H^2(A)$, $(\tilde{\psi}_n)_{n\in \N}\subset H^1(A \times I,\R^3)$ and a sequence of sets $(A_n)_{n\in \N}$  such that for each $n \in\N$, $A_n \ll A_{n+1} \ll A$ and $\cup_{n \in \N} A_n=A$ and
\begin{enumerate}[(a)]
\item $\tilde{\varphi}_n=0$, $\nabla' \tilde{\varphi}_n=0$  in a neighborhood of $\partial A$,  $\tilde{\psi}_n=0$ in a neighborhood of $\partial A \times I$.
\item $\tilde{\psi}_n=\psi_n \text{ on } A_n \times I,\ \tilde{\varphi}_n=\varphi_n \text{ on } A_n$,
\item $ \|\tilde{\varphi}_n-\varphi_n\|_{H^2}\to0,\  \|\tilde{\psi}_n-\psi_n\|_{H^1}\to 0$, $\|\nabla_{h_n} \tilde{\psi}_n-\nabla_{h_n} \psi_n\|_{L^2} \to 0$, as $n \to \infty$.
\item for $\tilde{\vartheta}_n$ defined by
$$ \tilde{\vartheta}_n:= \tilde{\psi}_{n}+\left(\begin{array}{c} 0\\ 0\\ \tfrac{\tilde{\varphi}_n}{h_n} \end{array} \right)-x_3 \left(\begin{array}{c} \partial_1 \tilde{\varphi}_n \\ \partial_2 \tilde{\varphi}_n \\ 0 \end{array} \right),$$
we have
$$ \lim_{n \to \infty}\left\| \sym \nabla_{h_n}\vartheta_n-\sym \nabla_{h_n}\tilde{\vartheta}_n \right\|_{L^2}=0.$$
\end{enumerate}
\end{lemma}
\begin{proof}
By $\theta:[0,+\infty) \to [0,+\infty)$ denote the function:
$$ \theta(\varepsilon)=\sup_{n \in \N,\ S \subset A \atop \meas(S) \leq \eps}\left(\|\nabla'^2 \varphi_n\|^2_{L^2(S)}+\| \nabla_{h_n} \psi_n \|^2_{L^2(S\times I)} \right). $$
By the equi-integrability property we have $\eta(\eps) \to 0$ as $\eps \to 0$.
For fixed $k \in \N$ choose $A_{k} \ll A$ open set with Lipschitz boundary such that $\meas(A\backslash \bar{A}_{k}) \leq \tfrac{1}{k}$ and smooth cut-off function $\eta_k \in C_0^\infty(A)$ such that $0 \leq \eta_k \leq 1$ and
$\eta_k=1$ in a neighborhood of $\bar{A}_{k}$. We can also assume  that for every $k \in \N$, $A_k \ll A_{k+1} \ll A$ and that $\cup_{k \in \N} A_k=A$.
Define $\tilde{\varphi}_{k,n}:= \eta_{k} \varphi_n$, $\tilde{\psi}_{k,n}:=\eta_{k} \psi_n$.
Define $g:\N \times [0,+\infty) \to [0,+\infty)$ by
$$g(k,n)= \|\tilde{\varphi}_{k,n}-\varphi_n\|_{H^2}+ \|\tilde{\psi}_{k,n}-\psi_n\|_{H^1}+\|\nabla_{h_n} \tilde{\psi}_{k,n}-\nabla_{h_n} \psi_n\|_{L^2}. $$
Since we have for $\alpha,\beta=1,2$:
\begin{eqnarray*}
 \partial_{\alpha \beta} \tilde{\varphi}_{k,n}&=&\partial_{\alpha \beta} \eta_{k} \varphi_n+\partial_{\alpha} \eta_{k} \partial_{\beta} \varphi_n+\eta_{k} \partial_{\alpha \beta} \varphi_n, \\
  \partial_{\alpha} \tilde{\psi}_{k,n}&=& \eta_{k} \partial_{\alpha} \psi_n+\partial_{\alpha} \eta_{k} \psi_n,\\
  \partial_3 \tilde{\psi}_{k,n}&=&\eta_{k} \partial_3 \psi_n,
\end{eqnarray*}
it is easy to conclude that there exists $C>0$ such that for every $k,n\in \N$ we have
$$g(k, n) \leq C\left(\theta(\tfrac{1}{k})+\|\eta_\eps\|_{C^2}\cdot(\|\varphi_n\|_{H^1}+\|\psi_n\|_{L^2})\right).$$
Since we also have, by the compactness, that $\varphi_n\to 0$, strongly in $H^1$ we conclude by the diagonalization argument that there exists a sequence $k(n)$ monotonly increasing  such that $g(k(n),n) \to 0$ as $n\to \infty$. This proves (c). (d) follows directly from (c).
\end{proof}

The following lemma is an easy consequence of Lemma \ref{lem:ocjena} and Lemma \ref{lem:glavnazamjena}.
It is crucial for proving Proposition \ref{identi}.
\begin{lemma}\label{lem:svojstva}
Let $(h_n)_{n \in \N}$ be such that  Assumption \ref{ass:main} is satisfied.
The following properties are valid for every $A,A_1,A_2\subset \omega$ open sets and  $M,M_1,M_2 \in L^2(\Omega,\R^{2 \times 2}_{\sym})$:
\begin{enumerate}[(a)]
\item
 there exists $K(M,A)$ such that for every $(h_n)_{n\in \N}$ monotonly decreasing to zero we have
$$
K^{+}_{(h_n)_{n \in \N}}(M,A)=K^{-}_{(h_n)_{n \in \N}}(M,A)=K(M,A),
$$
\item (localization) if $A_1 \subset A_2$  we have
$$K(M1_{A_1 \times I},A_2)=K(M,A_1), $$
\item (inner regularity)
$$K(M,A)=\sup_{D \in \mathcal{D} \atop D\ll A} K(M,D),$$
\item (boundedness)
$$K(M,A) \leq \beta \|M\|^2_{L^2(A \times I)},$$
\item (monotonicity) if $A_1 \subset A_2 $  we have
$$ K(M,A_1) \leq K(M,A_2),$$
\item (continuity) Let $(A_n)_{n\in \N}$ be the family of open subsets of $\omega$ such that for each $n\in \N$, $A_n \subset A_{n+1}$. Let $A=\cup_{n=1}^\infty A_n$. Then
    $\lim_{n\to \infty} K(M,A_n)=K(M,A)$,
\item (additivity) if $A_1 \cap A_2=\emptyset$ then we have
$$K(M,A_1 \cup A_2)=K(M,A_1)+K(M,A_2),$$
\item \label{neprekidnost1}
\begin{eqnarray*}
\left| K(M_1,A)-K(M_2,A)\right|&\leq& C \|M_1-M_2\|_{L^2}\left(\|M_1\|_{L^2}+\|M_2\|_{L^2}\right),
\end{eqnarray*}
where $C$ depends on $\alpha,\beta$,
\item (homogeneity)
    $$ K(tM,A) =t^2 K(M,A), \forall t \in \R. $$
\item
    $$ K(M_1+M_2,A)+K(M_1-M_2,A) \leq 2K(M_1,A)+2K(M_2,A), $$
\item (coerciveness) $$   K(M,A) \geq \alpha \|M\|^2_{L^2(A \times I)}. $$
\item (subadditivity) if   $A \subset A_1 \cup A_2$ we have
$$K(M,A) \leq K(M,A_1)+K(M,A_2).$$

\end{enumerate}

\end{lemma}
\begin{proof}
Using Lemma \ref{lem:glavnazamjena} it is easy to see that for $A \subset \omega$ open and an arbitrary $D \in \mathcal{D}$, $D \subset A$ we have
\begin{equation} \label{app100}
K^+_{(h_n)_{n\in \N}}(M1_{D\times I},A)=K^{-}_{(h_n)_{n\in \N}}(M1_{D \times I},A)=K(M,D),
\end{equation}
Namely, the inequality
$K^-_{(h_n)_{n\in \N}}(M 1_{D\times I},$ $A)$ $\geq$ $K(M,D)$, follows immediately from the definition in Remark \ref{minat2}. To prove the inequality
$K^+_{(h_n)_{n\in \N}}(M1_{D\times I},A) \leq K(M,D)$
it is enough to prove that for every $r>0$ we have $\limsup_{n \to \infty} K_n(M,D,A,B(r)) \leq K(M,D)$, where
$$ K_n (M,D,A,B(r))=\inf_{\psi \in H^1(A \times I,\R^3) \atop \|(\psi_1,\psi_2,h_n\psi_3)\|_{L^2} \leq r} \int_{A \times I} Q^{h_n}\left(x,\iota(M1_{D\times I})+\nabla_{h_n} \psi\right) \, dx.$$
To prove this take a subsequence, still denoted by $(h_{n(k)})$, where is $\limsup$  accomplished.
Then we take, using Lemma \ref{lem:glavnazamjena} a further subsequence, still denoted by $(h_{n(k)})_{k \in \N}$, and $(\chi_k)_{k \in \N} \subset H^1(D \times I,\R^3)$ such that $(\chi_{k,1},\chi_{k,2}, h_{n(k)} \chi_{k,3}) \to 0$, strongly in $L^2$, for each $k \in \N$, $\chi_k=0$ on $\partial D \times I$ and $$K(M,D)=\lim_{k\to \infty} \int_{D \times I} Q^{h_{n(k)}}(x,\iota(M)+\nabla_{h_{n(k)}} \chi_k)\, dx.$$
By extending $\chi_k=0$ on $(A\backslash D)\times I$ we obtain
\begin{eqnarray*}
K(M,D)&=&\lim_{k\to \infty} \int_{A \times I} Q^{h_{n(k)}}(x,\iota(M1_{D \times I})+\nabla_{h_{n(k)}} \chi_k)\, dx \\
& \geq & \lim_{k\to \infty} K_{n(k)} (M,D,A,B(r)).
\end{eqnarray*}
By the arbitrariness of $r$ we have the claim. (a) and (b)  follow from (\ref{app100}) and Lemma \ref{lem:ocjena} by an approximation argument i.e. by exhausting $A$ with the sets in $\mathcal{D}$.
It is easy to notice from the definition in Remark \ref{minat2} that $K(M,D) \leq K(M,A)$, for $D\in \mathcal{D}$, $D\subset A$. (c) then easily follows from (b) and  Lemma \ref{lem:ocjena}.
From the definition in Remark \ref{minat2}, by taking the null sequence, it is easy to see that for every $D \in\mathcal{D}$ we have $K(M,D) \leq \beta \|M\|^2_{L^2(D)}$. (d) now follows from (c).
 (e) easily follows from (c). (f) is again the direct consequence of (b) and  Lemma \ref{lem:ocjena}.
 To prove (g) first choose $D_1,D_2 \in \mathcal{D}$, $D_1\ll A_1$ and $D_2 \ll A_2$. We have that $D_1 \cap D_2=\emptyset$. From the definition in Remark \ref{minat2} it is easy to see that
 $$K(M,D_1 \cup D_2)=K(M,D_1)+K(M,D_2).$$
(g) now follows from (f).
 To prove (h) notice that from Lemma \ref{lem:ocjena} we can conclude that there exists
 $C>0$ dependent only on $\alpha,\beta$ such that for each  $D \in \mathcal{D}$ is valid
\begin{eqnarray}\label{ocjena111111}
\left| K(M_1,D)-K(M_2,D)\right|&\leq& C \|M_1-M_2\|_{L^2}\left(\|M_1\|_{L^2}+\|M_2\|_{L^2}\right),\\\nonumber & &  \ \forall M_1,M_2 \in L^2(\omega,\R^{2\times 2}_{\sym}).
\end{eqnarray}
 (h) follows from (f) and (\ref{ocjena111111}). To prove (i) we can take $D \in \mathcal{D}$ and $M \in L^2(\Omega,\R^{2 \times 2}_{\sym})$. Define by
 \begin{equation}
 K_n\left(M,D,B(r)\right)=\min_{\psi \in H^1(D \times I,\R^3) \atop \|(\psi_1,\psi_2,h_n\psi_3)\|_{L^2} \leq r}
 \int_{D \times I} Q^{h_n}\left(x,\iota(M)+\nabla_{h_n} \psi\right) \, dx.
 \end{equation}
The minimum in the above expression exists by the direct methods of the calculus of variation. Notice that
 $$K(M,D)=\lim_{r \to 0} \lim_{n \to \infty} K_n\left(M,D,B(r)\right), $$
 where limit in $n$ can be taken on any converging subsequence (dependent on $r$).
 Since every $Q^{h_n}$ is quadratic we have
$$t^2 K_n\left(M,D,B(r)\right)=K_n\left(tM,D,B(|t|r)\right). $$
 From this identity it follows $K(tM,D)=t^2K(M,D)$.
 By approximation we obtain (i).
 To prove (j) take $M_1,M_2 \in L^2(\Omega,\R^{2\times 2}_{\sym})$, $D \in \mathcal{D}$ and for $\alpha=1,2$, $\psi^{\alpha,r,n}\in H^1(D \times I,\R^3)$
 such that
 $$K_n\left(M_\alpha,D,B(r)\right)=\int_{D \times I} Q^{h_n}\left(x,\iota(M_{\alpha})+\nabla_{h_n} \psi^{\alpha,r,n}\right) \, dx.   $$
 and $\|(\psi^{\alpha,r,n}_1,\psi^{\alpha,r,n}_2,h_n\psi^{\alpha,r,n}_3)\|_{L^2} \leq r$.
 Notice that:
 \begin{eqnarray*}
& & K_n\left(M_1+M_2,D,B(2r)\right)+K_n\left(M_1-M_2,D,B(2r)\right)   \\ & &\hspace{8ex} \leq \int_{D \times I} Q^{h_n}\left(x,\iota(M_1+M_2)+\nabla_{h_n} \psi^{1,r,n}+\nabla_{h_n} \psi^{2,r,n}\right) \, dx\\ & &\hspace{10ex}+ \int_{D \times I} Q^{h_n}\left(x,\iota(M_1-M_2)+\nabla_{h_n} \psi^{1,r,n}-\nabla_{h_n} \psi^{2,r,n}\right) \, dx\\ &&\hspace{8ex}=
 2\int_{D \times I} Q^{h_n}\left(x,\iota(M_1)+\nabla_{h_n} \psi^{1,r,n}\right) \, dx+2\int_{D \times I} Q^{h_n}\left(x,\iota(M_2)+\nabla_{h_n} \psi^{2,r,n}\right)\\ &&\hspace{8ex}=2K_n(M_1,D,B(r))+2K_n(M_2,D,B(r)),
 \end{eqnarray*}
where we have used (e) of Proposition \ref{kvforme}.
By letting $n \to \infty$ and then $r \to 0$ we obtain that
$$K(M_1+M_2,D)+K(M_1-M_2,D) \leq 2K(M_1,D)+2K(M_2,D). $$
(j) follows by density and (f). To prove (k) take $M \in C^1(\Omega,\R^{2 \times 2}_{\sym})$,
such that $M=0$ in a neighborhood of $\partial \Omega$,
$D \in \mathcal{D}$ and
$(\chi_n)_{n \in \N} \subset H^1(D \times I,\R^3)$ such that
$(\chi_{n,1},\chi_{n,2}, h_{n} \chi_{n,3}) \to 0$, strongly in $L^2$ and such that
$$K(M,D)=\lim_{n\to \infty} \int_{D \times I} Q^{h_{n}}(x,\iota(M)+\nabla_{h_{n}} \chi_n)\, dx.$$
From (Q1) we conclude
\begin{eqnarray*}
 \int_{D \times I} Q^{h_{n}}(x,\iota(M)+\nabla_{h_{n}} \chi_n)\, dx &\geq&
 \alpha \int_{D \times I} \left|M+\sym \nabla' (\chi_{n,1}\, ,\, \chi_{n,2} )\right|^2 \, dx \\ & \geq&
\alpha \|M\|^2_{L^2}-\alpha \int_{D \times I} \diver M \cdot (\chi_{n,1}\, , \, \chi_{n,2}) \, dx.
\end{eqnarray*}
By letting $n \to \infty$ and using the fact that $(\chi_{n,1}\, , \, \chi_{n,2}) \to 0$ strongly in $L^2$ we have that $K(M,D) \geq \alpha \|M\|^2_{L^2}$. (k) follows from (f) and (h).
To prove (l) we proceed as follows;
by using  (c) and (e)  it is enough to prove that for arbitrary $D,D_1,D_2 \in \mathcal{D}$ such that $D \subset D_1 \cup D_2$ we have
$$ K(M,D) \leq K(M,D_1)+K(M,D_2). $$
Take $D_2'\in\mathcal{D}$ such that $D_2'\ll D_2\backslash \bar{D}_1$.
From (f)  we have
$$K(M,D_1 \cup D_2')=K(M,D_1)+K(M,D_2'). $$
From (g)   we have for some $C>0$
\begin{eqnarray*}
& &K(M,D)=K(M1_{D \times I},\omega)\\ & & \leq K(M1_{(D_1 \cup D_2') \times I},\omega)+C\|M\|_{L^2\left( (D_2\backslash(D_1 \cup D_2')\times I\right)}\|M\|_{L^2}\\ & &=K(M,D_1 \cup D_2')+C\|M\|_{L^2\left( (D_2\backslash(D_1 \cup D_2')\times I\right)}\|M\|_{L^2}\\ & &=K(M,D_1)+K(M,D_2')+C\|M\|_{L^2\left( (D_2\backslash(D_1 \cup D_2')\times I\right)}\|M\|_{L^2}
\\ & &\leq K(M,D_1)+K(M,D_2)+C\|M\|_{L^2\left( (D_2\backslash(D_1 \cup D_2')\times I\right)}\|M\|_{L^2}.
\end{eqnarray*}
The claim follows by the arbitrariness of $D_2'$.

\end{proof}

At the end of this section  we improve Lemma \ref{lem:glavnazamjena} for arbitrary $A \subset \omega$ open.
\begin{lemma}\label{lem:improveglavnazamjena}
 Let Assumption \ref{ass:main} be satisfied for a sequence $(h_n)_{n \in \N}$ monotonly decreasing to zero.
 Take $M \in  L^2(\Omega,\R^{2 \times 2}_{\sym})$ and $A \subset \omega$ open.
Then there exists a subsequence $(h_{n(k)})_{k\in \N}$ and $(\vartheta_k)_{k \in \N} \subset H^1(A \times I,\R^3)$ such that
\begin{enumerate}[(a)]
\item
$(\vartheta_{k,1},\vartheta_{k,2},h_{n(k)} \vartheta_{k,3}) \to 0$ strongly in $L^2$,
\item
 $(|\sym \nabla_{h_{n(k)}} \vartheta_k|^2 )_{k \in \N}$ is equi-integrable,
 $$ \sym \nabla_{h_{n(k)}} \vartheta_k=-x_3\iota(\nabla'^2 \varphi_{k})+\sym \nabla_{h_{n(k)}} \tilde{\psi}_{k},$$
 where $\left(|\nabla^2 \varphi_k|^2\right)_{k\in\N}$ and $\left(| \nabla_{h_{n(k)}}\tilde{\psi}_k |^2\right)_{k \in \N}$ are equi-integrable and
 $\varphi_k \to 0$ strongly in $H^1$ and $\tilde{\psi}_k \to 0$ strongly in $L^2$.
 Also the following is valid
 $$\limsup_{k \to \infty} \left( \| \varphi_k \|_{H^2(A)}+\|\nabla_{h_{n(k)}}{\tilde{\psi}}_k\|_{L^2(A)}\right)\leq C \left(  \|M\|^2_{L^2}+1 \right),$$
 where $C$ is independent of the domain $A$ and for each $k \in \N$ we have $\vartheta_k=0$ in a neighborhood of $\partial A \times I$, i.e., $\varphi_k=\nabla' \varphi_k=0$, in a neighborhood of $\partial A$ and $\tilde{\psi}_k=0$ in a neighborhood  of $\partial A \times I$.
\item $$K(M,A)=\lim_{k\to \infty} \int_{A \times I} Q^{h_{n(k)}}(x,\iota(M)+\nabla_{h_{n(k)}} \vartheta_k)\, dx.$$

\end{enumerate}
\end{lemma}
\begin{proof}
Take $r>0$ such that $B(r) \supset \omega$. Extend $Q^h$ on $(B(r) \backslash \omega)\times I$ by e.g. $Q^h(x,G)=\beta |\sym G|^2$, for all $x \in (B(r) \backslash \omega) \times I$. Apply Lemma \ref{lem:glavnazamjena} to $\tilde{M}=M1_{A}$ and $D=B(r)$ to obtain the sequences $(\tilde{\vartheta}_k)_{k \in \N} \subset H^1(B(r) \times I,\R^3)$, $(\tilde{\varphi}_k)_{k \in N} \subset H^2(B(r))$, $(\tilde{\tilde{\psi}}_k)_{k \in \N} \subset H^1(B(r) \times I,\R^3)$ that satisfy (a), (b), (d) of Lemma \ref{lem:glavnazamjena}. In the same way as in Lemma \ref{nulizacija} for each $\eps>0$ we choose $A_\eps \ll A$ with Lipschitz boundary such that $\meas (A \backslash A_\eps)<\varepsilon$ and a cut off function $\eta_\epsilon \in C_0^\infty(A)$ which is $1$ on $\bar{A}_{\eps}$. Again by using the diagonal procedure, we obtain a sequence $(\varphi_k)_{k \in \N} \subset H^2(A)$ and $(\tilde{\psi}_k)_{k \in \N} \subset H^1(A\times I,\R^3)$ such that for each $k \in \N$, $\varphi_k=\nabla \varphi_k=0$  in a neighborhood of  $B(r) \backslash A$, i.e., $\tilde{\psi}_k=0$ in a neighborhood of $(B(r) \backslash A) \times I$ and
$$ \lim_{k \to \infty} \left(\|\tilde{\varphi}_k-\varphi_k\|_{H^2(A)}+\|\tilde{\tilde{\psi}}_k-\tilde{\psi}_k\|_{L^2(A \times I)}+\|\nabla_{h_{n(k)}}\tilde{\tilde{\psi}}_k-\nabla_{h_{n(k)}}\tilde{\psi}_k\|_{L^2(A \times I)}\right)=0.$$
Define again
$$\vartheta_k:= \tilde{\psi}_{k} +\left(\begin{array}{c} 0\\ 0\\ \tfrac{\varphi_k}{h_{n(k)}} \end{array} \right)-x_3 \left(\begin{array}{c} \partial_1 \varphi_k \\ \partial_2 \varphi_k \\ 0 \end{array} \right).$$
It is easy to see from (b) of Lemma \ref{lem:svojstva} that
\begin{eqnarray*}
K(M,A)&=&K(M1_{A \times I},B(r))=\lim_{k \to \infty} \int_{B(r) \times I} Q^{h_{n(k)}}(x,\iota(M 1_{A\times I})+\nabla_{h_{n(k)}} \tilde{\vartheta}_k)\, dx \\ &\geq& \lim_{k \to \infty} \int_{A \times I} Q^{h_{n(k)}}(x,\iota(M)+\nabla_{h_{n(k)}} \vartheta_k)\, dx \geq K(M,A).
\end{eqnarray*}
From this we have the claim.
\end{proof}

\subsection{Proof of Proposition \ref{identi} and Lemma \ref{lem:pojasnjenje} }

\begin{proof}[Proof of Proposition \ref{identi}]
The proof follows the standard  steps in $\Gamma$-convergence theory. \\
Notice that for $M \in \mathcal{S}_{vK}(\omega)$
\begin{equation} \label{anna1}
 \|M\|^2_{L^2(\Omega)}=\|M_1\|^2_{L^2(\omega)}+\tfrac{1}{12}\|M_2\|^2_{L^2(\omega)}.
 \end{equation}
\step 1 Existence of $Q$. \\
From Theorem \ref{DMtm}  as well as the properties (d), (e), (g), (f) and (l) from Lemma \ref{lem:svojstva} we conclude from Radon-Nykodim theorem that for an arbitrary $M \in \mathcal{S}_{vK}(\omega)$ there exists $Q_M \in L^1(\omega)$, a positive function, such that
\begin{equation} \label{eq:postojiiii}
 K(M,A)= \int_A Q_M(x') \, dx', \forall A \subset \omega \text{ open}.
\end{equation}
Take a countable dense subset $\mathcal{M}$ of $\R^{2 \times 2}_{\sym}$ and define
$$E:=\{x' \in \omega:\, x' \text{ is a Lebesgue point for } Q_{M_1+x_3M_2} \text{ for every } M_1,M_2 \in \mathcal{M} \}.$$
Notice that $\meas(\omega \backslash E)=0$.
Define also
 \begin{eqnarray} \label{defQ}
&& Q(x',M_1,M_2)=Q_{M_1+x_3 M_2}(x')=\lim_{r \to 0} \tfrac{1}{|B(x',r)|} K\left(M_1+x_3M_2,B(x',r)\right),\\ \nonumber &&\hspace{10ex} \text{ for } M_1,M_2 \in \mathcal{M} \text{ and } x' \in E.
 \end{eqnarray}
Notice that from the property (h) in Lemma \ref{lem:svojstva} we have
\begin{eqnarray*}
&& |Q(x',M_1,M_2)-Q(x',M_1',M_2')| \leq \\ && \hspace{5ex} C\left(|M_1-M_1'|+|M_2-M_2'| \right)(|M_1+M_1'|+|M_2+M_2'|),\\ && \hspace{10ex} \text{ for all } M_1, M_1', M_2, M_2'  \in \mathcal{M},\, x' \in E.
\end{eqnarray*}
Thus we can extend $Q(\cdot,\cdot,\cdot)$ by continuity on $E \times \R^{2 \times 2}_{\sym}\times \R^{2 \times 2}_{\sym}$. Extend it by zero on $\omega \times \R^{2 \times 2}_{\sym}\times$ $\R^{2 \times 2}_{\sym}$. Notice that such defined $Q$ satisfies
\begin{eqnarray}
\label{svQ1}
&& |Q(x',M_1,M_2)-Q(x',M_1',M_2')| \leq \\ \nonumber && \hspace{5ex} C\left(|M_1-M_1'|+|M_2-M_2'| \right)(|M_1+M_1'|+|M_2+M_2'|),\\ \nonumber && \hspace{10ex} \text{ for all } M_1, M_1', M_2, M_2'  \in \R^{2 \times 2}_{\sym},\, x' \in \omega,
 \\
\label{svQ2} & &|Q(x',M_1,M_2)| \leq \beta \left(|M_1|^2+|M_2|^2 \right), \text{ for all } M_1,M_2  \in \R^{2 \times 2}_{\sym},\, x' \in \omega.
\end{eqnarray}
Also, from the property (h) in Lemma \ref{lem:svojstva} and (\ref{svQ1}) we conclude that
\begin{equation}\label{drugadefQ}
Q(x',M_1,M_2)=\lim_{r \to 0} \tfrac{1}{|B(x',r)|} K\left(M_1+x_3M_2,B(x',r)\right), \forall M_1,M_2 \in \R^{2 \times 2}_{\sym} \text{ and } x' \in E.
\end{equation}
By approximating $M \in \mathcal{S}_{vK}(\omega)$ by piecewise constant maps with values in the set $\{M_1+x_3 M_2: M_1,M_2 \in \mathcal{M} \} $ we conclude from (b), (g), (h) of Lemma \ref{lem:svojstva} as well as the properties (\ref{svQ1}) and (\ref{svQ2})  that
\begin{equation}
K(M,\omega)=\int_\omega Q(x',M_1(x'),M_2(x')) \, dx', \forall M\in \mathcal{S}_{vK}(\omega).
\end{equation}
By using (b) of Lemma \ref{lem:svojstva} as well as the fact that $Q(x',0,0)=0$ $\forall x' \in \omega$ we conclude (\ref{eq:postojiiii}).\\
\step 2 Quadraticity and coercivity of $Q$. \\
To prove that $Q$ is quadratic form we use (\ref{drugadefQ}), Proposition \ref{kvforme} and (i), (j) of Lemma \ref{lem:svojstva}. To prove coercivity we use (\ref{anna1}), (\ref{drugadefQ}) and property (k) of Lemma \ref{lem:svojstva}.
\end{proof}
\begin{proof}[Proof of Lemma \ref{lem:pojasnjenje}]
By Remark \ref{razjasnjenje} and Lemma \ref{podniz1} it is enough to see that every sequence $(h_n)_{n \in \N}$ monotonly decreasing to zero has subsequence such $(h_{n(k)})_{k \in \N}$ such that
$$K^{-}_{(h_{n(k)})_{k \in \N}}(M,D)=:K(M,D), \ \forall n \in \N, M \in \mathcal{S}_{vK}(\omega), D \in \mathcal{D},$$
where $K(M,D)$ is independent of the sequence.
This follows from Lemma \ref{podniz}, Remark \ref{razjasnjenje} and Proposition \ref{identi}, i.e., (\ref{eq:postojiiii}) and (\ref{defQ}).
\end{proof}

\subsection{Proof of Theorem \ref{thm:1}}
The proof of the following proposition is given in \cite[Proposition 3.1]{NeuVel-12}. It characterizes the deformations which have the order of the energy $h^4$.
\begin{proposition}
  \label{P1}
  Let $\ya\in H^1(\Omega,\R^3)$ and $h>0$.
  There exist  $(\bar\Ra,\ua,v)\in \SO 3\times H^1(\omega,\R^2)\times H^2_\loc(\omega)$ and correctors
  $w\in H^1(\omega)$, $\psi\in H^1(\Omega,\R^3)$ with
  \begin{gather*}
    \int_\omega w=0,\qquad\int_I\psi(\hat
    x,x_3)\,dx_3=0\quad\text{ for almost every }x'\in\omega,
  \end{gather*}
  such that
  \begin{equation}\label{P1.0}
    \bar{\Ra}^\transpose\left(\ya-\fint_\Omega\ya\ud x\right)=
    \left(\begin{array}{c}
        x'\\ hx_3
      \end{array}\right)
    +
    \left(\begin{array}{c}
        h^2\ua\\ h (v+h w)
      \end{array}\right)
    -
    h^2x_3
    \left(\begin{array}{c}
        \hat\nabla v\\ 0
      \end{array}\right)
    +
    h^2\psi
  \end{equation}
  and
  \begin{equation}\label{P1.1}
    \|\ua\|^2_{H^1(\omega)}+\|v\|^2_{H^1(\omega)}+\|w\|^2_{H^1(\omega)}+\frac{1}{h^2}\|\psi\|^2_{L^2(\Omega)}\leq
    C(\omega)(e_h(\ya)+e_h(\ya)^2).
  \end{equation}

  In addition, for all $D\ll\omega$ compactly contained in $\omega$ we have
  \begin{equation}\label{P1.2}
    \|\nabla^2v\|^2_{L^2(D)}+\|\nabla_h\psi\|^2_{L^2(D\times
      I)}\leq C(D) (
    e_h(\ya)+e_h(\ya)^2).
  \end{equation}

If the boundary of $\omega$ is of class $C^{1,1}$, then
  $(\ua,v)\in\mathcal A(\omega)$ and \eqref{P1.2} holds for $D$ replaced by $\omega$.
\end{proposition}

The following lemma is an easy consequence of Taylor expansion. It is the essential part of lower bound theorem.
\begin{lemma}
  \label{lem:3}
  Let   $\bs G^h,\bs K^h\in L^2(\Omega,\R^{3 \times 3})$ be such that
  \begin{align}
    \label{eq:1}
    &\bs K^h\text{ is skew-symmetric and}\\
    &\limsup\limits_{h\to 0}\left(\|\bs G^h\|_{L^2}+\|\bs K^h\|_{L^4}\right)<\infty,
  \end{align}
  Consider
  \begin{equation*}
    \bs E^h:=\frac{\sqrt{(\id+h\bs K^h+h^2\bs G^h)^t(\id+h\bs K^h+h^2\bs G^h)}-\id}{h^2}.
  \end{equation*}
 Then there exists a sequence $(\chi^h)_{h>0}$ such that, $\chi^h:\Omega \to \{0,1\}$,
  $\chi^h \to 1$ boundedly in measure and
  \begin{equation*}
    \lim\limits_{h\to 0}\left\|\chi^h\left(
    E^h-\left(\sym\bs G^h-\frac{1}{2}(\bs K^h)^2\right) \right) \right\|_{L^2}=0.
  \end{equation*}
\end{lemma}
\begin{proof}
Notice that the following claim is the direct consequence of Taylor expansion: There exists $\delta>0$ and a monotone increasing function $\eta:(0,\delta) \to (0,\infty)$ such that $\lim_{r \to 0} \eta(r)=0$ and
\begin{eqnarray} \label{igor100}
\left| \sqrt{(I+A)^t (I+A)}-\left(I+\sym A+\tfrac{1}{2} A^t A\right)\right| &\leq& \eta(|A|)\left(\sym A+\tfrac{1}{2} A^t A\right) \\ \nonumber & & \forall A\in \R^{3 \times 3}, |A|<\delta.
\end{eqnarray}
Now we use the truncation argument. Namely, let $\chi^h$ be the characteristic function of the set $S^h$,
where
$$S^h=\{x \in \Omega: |K^h| \leq \tfrac{1}{\sqrt{h}},\ |G^h| \leq \tfrac{1}{h} \}. $$
The claim follows after putting $A=hK^h+h^2G^h$ into the expression (\ref{igor100}), dividing by $h^2$ and letting $h \to 0$.
\end{proof}

We state one simple linearization lemma, which is already given in \cite{NeuVel-12}.
\begin{lemma}[linearization]
  \label{L:linearization}
  Let $\{\widetilde E^h\}_{h>0}\subset L^2(\Omega,\R^{3\times 3})$
  satisfy
  \begin{equation}\label{eq:19}
    \limsup\limits_{h\to 0}\|\widetilde E^h \|_{L^2}<\infty\qquad\text{ and }\qquad \lim\limits_{h\to 0}h^2\|\widetilde E^h\|_{L^\infty}=0.
  \end{equation}
  Then
  \begin{equation*}
  \lim\limits_{h\to 0} \left| \frac{1}{h^4}\int_\Omega
      W^h(x,\id+h^2\widetilde E^{h}(x))\,dx-\int_\Omega Q^h(x,\widetilde E^{h}(x))\,dx\right|=0.
  \end{equation*}
\end{lemma}
\begin{proof}
  We have
  \begin{eqnarray*}
    \limsup_{h \to 0} \lefteqn{\left| \frac{1}{h^4}\int_\Omega
      W^h(x,\id+h^2\widetilde  E^{h}(x))\,dx-\int_\Omega
      Q^h(x,\widetilde E^{h}(x))\,dx\right|}&&\\
  &\leq& \limsup_{h \to 0}
    \frac{1}{h^4}\int_\Omega\left|W^h(x,\id+h^2\widetilde E^{h}(x))- h^4Q^h(x,\widetilde E^{h}(x))\right|\,dx\\
    &\stackrel{\eqref{eq:94}}{\leq}& \limsup_{h \to 0}
    \frac{1}{h^4}\int_\Omega|h^2\widetilde E^{h}(x)|^2\,r(|h^2\widetilde E^{h}(x)|)\,dx\\
    &\leq& \limsup_{h \to 0} r(h^2\|\widetilde E^{h}\|_{L^\infty})\int_\Omega|\widetilde E^{h}(x)|^2\,dx,
  \end{eqnarray*}
  where in the last line we used that $r(\cdot)$ is monotonically
  increasing. By appealing to \eqref{eq:19} and $\lim_{\delta\to
    0}r(\delta)=0$, we get
  $ \lim_{h\to0}r(h^2\|\widetilde E^h\|_{L^\infty})\int_\Omega|\widetilde E^{h}(x)|^2\,dx=0 $
  and the proof is complete.
\end{proof}

\begin{proof}[Proof of Theorem \ref{thm:1}]
The key fact is to obtain from Proposition \ref{P1} the representation of the strain in the form
\begin{equation}
 E^{h_n}=\sym \nabla u-\frac{1}{2} \nabla v \otimes \nabla v-x_3 \nabla^2 v+\nabla'^2\varphi^{h_n}+\nabla_{h_n}\tilde{\psi}^{h_n}+o^{h^n},
 \end{equation}
on a large set that vanishes as $n\to\infty$. Here
\begin{eqnarray*}
& & \lim_{n \to \infty} \|o^{h_n}\|_{L^2}=\lim_{n \to \infty}\|\tilde{\psi}^{h_n}\|_{L^2}=\lim_{n\to \infty}\| \varphi^{h_n}\|_{H^1}=0, \\
& & \limsup_{n \to \infty}\|\varphi^{h_n}\|_{H^2} < \infty, \quad \limsup_{n \to \infty} \|\nabla_{h_n} \tilde{\psi}^{h_n} \|_{L^2}<\infty.
\end{eqnarray*}
This is established by the  relations (\ref{eq:elias1}), (\ref{eq:elias2}), (\ref{konacno1}) and (\ref{konacno2}). To make the lower bound we have to modify the sequences  $(|\nabla'^2\varphi^{h_n}|^2)_{n\in\N}$ and $(|\nabla_{h_n}\tilde{\psi}^{h_n}|^2)_{n\in\N}$ by equi-integrable ones.

For the proof without loss of generality we can take $R=I$.
First we assume that $\omega$ is of class $C^2$.
Without loss of generality we assume that
\begin{equation}\label{eq:22}
  \liminf_{n\to \infty}I^{h_n}(\ya^{h_n})=\limsup_{n\to
    \infty}I^{h_n}(\ya^{h_n})<\infty.
\end{equation}
Due to the non-degeneracy of $W^{h_n}$ (see
\eqref{ass:non-degenerate}) we have
$\limsup_{n\to\infty}e^{h_n}(\ya^{h_n})<\infty$. Hence, Proposition~\ref{P1}
is applicable, and we easily deduce the part (i), by taking in the expression (\ref{P1.0}) the integral over the interval $I$.
From the estimate (\ref{P1.1}) and (\ref{P1.2}) we conclude that
$$ v^{h_n} \rightharpoonup v \text{ weakly in } H^2,u^{h_n} \rightharpoonup u \text{ weakly in } H^1. $$
Notice that from the expression (\ref{P1.0}) we have that
$$ \nabla_{h_n} y^{h_n}= I+{h_n}K^{h_n}+h_n^2 G^{h_n},$$
where
\begin{eqnarray*}
K^{h_n}&=&\left(\begin{array}{ccc} 0 &0 & -\partial_1 v^{h_n} \\ 0 & 0 &-\partial_2 v^{h_n} \\ \partial_1 v^{h_n} & \partial_2 v^h & 0  \end{array} \right), \\
G^{h_n} &=& \iota (-x_3 \nabla'^2 v^{h_n}+\nabla' u^{h_n})+\nabla_{h_n} \psi^{h_n}+\left(\begin{array}{ccc} 0 & 0 &0 \\ 0& 0 &0 \\ \partial_1 w^{h_n} & \partial_2 w^{h_n} & 0 \end{array} \right).
\end{eqnarray*}
Notice that $K^{h_n}$, $G^{h_n}$ satisfy the hypothesis of Lemma $\ref{lem:3}$.
Define also $$E^{h_n}=\tfrac{\sqrt{(\nabla_{h_n} y^{h_n})^t \nabla_{h_n} y^{h_n}}-I}{h_n^2}.$$
From the inequality valid for any $F \in \R^{3 \times 3}$, $|\sqrt{F^tF}-I| \leq \dist(F,\SO 3)$, where the equality holds for $F \in \R^{3 \times 3}$ such that $\det F>0$, we conclude
that $\limsup_{n \to \infty} \|E^{h_n}\|_{L^2} < \infty$.
We truncate the peaks of
$E^h$ and the set of points where $\det\nabla_{h_n}\ya^{h_n}$ is
negative. Therefore, consider the good set
$C^{h_n}:=\{\,x\in\Omega\,:\,|E^{h_n}(x)|\leq {h_n}^{-1},\,\det
\nabla_{h_n}\ya^{h_n}(x)>0\,\}$ and let $\chi^{h_n}_1$ denote the indicator
function associated with $C^{h_n}$. It is easy to see that $\chi^{h_n}_1 \to 1$ boundedly in measure.
By applying Lemma \ref{lem:3} we know that there exists a sequence $(\chi_2^{h_n})_{n \in\N}$ such that for all $n$, $\chi_2^{h_n}:\Omega \to \{0,1\}$,
  $\chi_2^{h_n} \to 1$ boundedly in measure and
  \begin{equation} \label{eq:elias1}
    \lim\limits_{n\to \infty}\left\|\chi_2^{h_n}\left(
    E^{h_n}-\left(\sym\bs G^{h_n}-\frac{1}{2}(\bs K^{h_n})^2\right) \right) \right\|_{L^2}=0.
  \end{equation}

In the same
way as in Proposition \ref{igor2} we take $(\tilde{w}^{h_n})_{n \in\N}$ such that
\begin{equation} \label{igor1222}
 \lim_{n \to \infty} \|w^{h_n}-\tilde{w}^{h_n}\|_{L^2}=0,\ \limsup_{n \to \infty} \|\tilde{w}^{h_n}\|_{H^1}\leq C(\omega)\limsup_{n \to \infty} \|w^{h_n}\|_{H^1} ,\ \lim_{n \to \infty} {h_n} \|\tilde{w}^{h_n}\|_{H^2}=0.
\end{equation}
Also we take the sequence $(\tilde{v}^{h_n})_{n \in \N} \subset C^2(\omega)$ such that
\begin{equation} \label{igor1223}
\|\tilde{v}^{h_n}-v\|_{H^2}\to 0, \quad {h_n}\| \tilde{v}^{h_n}\|_{C^2} \to 0.
\end{equation}
This can be done by taking a smooth sequence converging to $v$ and reparametrizing it.
Notice that
\begin{eqnarray}
\label{eq:elias2}& &\\  \nonumber\sym G^h -\tfrac{1}{2} (K^{h_n})^2 &=& \iota\left( M_1+x_3 M_2 \right)-x_3\iota \left( \nabla'^2 (v^{h_n}-v) \right)\\ \nonumber & &+\sym\nabla_{h_n} \tilde{\psi}^{h_n}+o^{h_n},
\end{eqnarray}
where
\begin{eqnarray*}
M_1 &=& \sym \nabla' u -\tfrac{1}{2} \nabla' v \otimes \nabla' v,\\
M_2 &=& -\nabla'^2 v, \\
\tilde{\psi}^{h_n} &=& \psi^{h_n}+\left(\begin{array}{c} u_1^{h_n}-u_1 \\ u_2^{h_n}-u_2 \\ w^{h_n}-\tilde{w}^{h_n} \end{array} \right)+
hx_3 \left(\begin{array}{c} \partial_1 \tilde{w}^{h_n} \\ \partial_2 \tilde{w}^{h_n}\\  -\tfrac{1}{2}\left(|\partial_1 \tilde{v}^{h_n}|^2+|\partial_2 \tilde{v}^{h_n}|^2\right)  \end{array}\right),\\
o^{h_n} &=& -\tfrac{1}{2} \iota\left( \nabla' v^{h_n} \otimes \nabla' v^{h_n}- \nabla' v \otimes \nabla' v \right)-h_nx_3\iota(\nabla'^2 \tilde{w}^{h_n})\\ & &+\tfrac{1}{2} \sym \left(\begin{array}{cc}  0 & 0 \\ 0 & 0 \\ h_nx_3 \nabla' \left(|\partial_1 \tilde{v}^{h_n}|^2+|\partial_2 \tilde{v}^{h_n}|^2  \right) & |\partial_1 \tilde{v}^{h_n}|^2+|\partial_2 \tilde{v}^{h_n}|^2-|\partial_1 v^{h_n}|^2-|\partial_2 v^{h_n}|^2 \end{array} \right).
\end{eqnarray*}
Notice that from (\ref{P1.1}), (\ref{P1.2}) as well as (\ref{igor1222}) and (\ref{igor1223}) and Sobolev embedding we conclude that
\begin{eqnarray}
& & \label{konacno1}\lim_{n \to \infty} \|o^{h_n}\|_{L^2}=\lim_{n \to \infty}\|\tilde{\psi}^{h_n}\|_{L^2}=\lim_{n\to \infty}\| v^{h_n}-v\|_{H^1}=0, \\
& & \label{konacno2} \limsup_{n \to \infty}\|\tilde{v}^{h_n}-v^{h_n}\|_{H^2} < \infty, \quad \limsup_{n \to \infty} \|\nabla_{h_n} \tilde{\psi}^{h_n} \|_{L^2}<\infty.
\end{eqnarray}
By using Proposition \ref{ekvi1} and Theorem \ref{ekvi2} we find a subsequence $(h_{n(k)})$ and sequences
$(\varphi_k)_{k \in \N} \subset H^2(\omega)$, $(\tilde{\tilde{\psi}}_k)_{k \in \N} \subset H^1(\Omega,\R^3)$, $(\chi_{3,k})_{k \in \N}$ such that
\begin{eqnarray}
& &\label{svojstvo11} \lim_{k \to \infty} \| \varphi_k\|_{H^1}= \lim_{k \to \infty} \|\tilde{\tilde{\psi}}_k\|_{L^2}=0, \\
& & \label{svojstvo22} \left( |\nabla'^2 \varphi_k|^2\right)_{k \in \N},\ \left( |\nabla_{h_{n(k)}} \tilde{\tilde{\psi}}_k|^2 \right)_{k \in \N} \text{ are equi-integrable}, \\
& & \label{svojstvo33} \chi_{3,k}:\Omega \to \{0,1\}, \forall k, \ \chi_{3,k}\to 1 \text{ boundedly in measure},\\
& & \label{svojstvo44}
\{ x=(x',x_3) \in \Omega: \varphi_k(x') \neq (v^{h_{n(k)}}-v)(x')  \text{ or } \tilde{\tilde{\psi}}_k(x) \neq \tilde{\psi}^{h_{n(k)}}(x)  \}=\{\chi_{3,k}=0\}.
\end{eqnarray}
Define
$$\vartheta_k=\tilde{\tilde{\psi}}_{k}+\left(\begin{array}{c} 0\\ 0\\ \tfrac{\varphi_k}{h_{n(k)}} \end{array} \right)-x_3 \left(\begin{array}{c} \partial_1 \varphi_k \\ \partial_2 \varphi_k \\ 0 \end{array} \right),$$
and notice that
$$ \sym \nabla_{h(n(k))} \vartheta_k= -x_3 \iota(\nabla'^2 \varphi_k)+\sym \nabla_{h(n(k))}\tilde{\tilde{\psi}}_k. $$
From this and (\ref{svojstvo22}) and (\ref{svojstvo44}) we conclude that
the family $(|\sym \nabla_{h(n(k))} \vartheta_k|^2)_{k \in \N}$ is equi-integrable and
 \begin{equation} \label{elias3}
 \{x \in \Omega: \ \sym \nabla_{h_{n(k)}} \vartheta_k \neq -x_3 \iota(\nabla'^2(v^{h_{n(k)}}-v))+\nabla'_{h_{n(k)}} \tilde \psi^{h_{n(k)}} \}=\{\chi_{3,k}=0\},
 \end{equation}
 up to a set of measure zero (see Remark \ref{rem:jednakost}).
 From (\ref{svojstvo11}) we conclude that $(\vartheta_{k,1},\vartheta_{k,2},h_{n(k)}\vartheta_{k,3})$ $\to 0$ strongly in $L^2$.
Now we are ready to prove the lower bound. The key idea is that the replacement by equi-integrable family enables us to establish the lower bound on the whole set.
Denote by
$$\chi_k=\chi_1^{h_{n(k)}} \chi_2^{h_{n(k)}} \chi_{3,k},\quad \widetilde E_k=\chi_k E^{h_{n(k)}}.$$
By appealing to the polar
factorization for matrices with non-negative determinant, there exists
a matrix field $\Ra^{h_n}:C^{h_n}\to\SO 3$ such that
\begin{equation*}
  \forall x\in C^{h_n}\,:\, \nabla_{h_n}\ya^{h_n}(x)=\Ra^{h_n}(x)\sqrt{(\nabla_{h_n}\ya^{h_n}(x))^t\nabla_{h_n}\ya^{h_n}(x)}.
\end{equation*}
Hence, by frame-indifference (see \eqref{ass:frame-indifference}),
non-negativity (see \eqref{ass:non-degenerate}) and
assumption \eqref{ass:stressfree} we
have
\begin{eqnarray*}
  W^{h_{n(k)}}(x,\nabla_{h_{n(k)}}\ya^{h_{n(k)}}(x))&\geq& \chi_{k}(x) W(x,\nabla_h\ya^h(x))\\ &=& W^{h_{n(k)}}(x,\id+h_{n(k)}^2 \widetilde E_k(x)).
\end{eqnarray*}
Thus,
\begin{eqnarray*}
  I^{h_{n(k)}}(\ya^{h_{n(k)}})&=&\frac{1}{h_{n(k)}^4}\int_\Omega W^{h_{n(k)}}(x,\nabla_{h_{n(k)}}\ya^{h_{n(k)}}(x))\,dx\\
  &\geq&\frac{1}{h_{n(k)}^4}\int_{\Omega} W^{h_{n(k)}}(x,\id+h_{n(k)}^2\widetilde E_k(x)\,\big)\,dx.
\end{eqnarray*}
Due to the truncation we have
$ \lim_{ k \to \infty }{h_{n(k)}}^2 \|\widetilde{E}_k\|_{L^\infty}=0$. Hence, using (\ref{ocjenakv1}),
Lemma~\ref{L:linearization} and the equi-integrability of   $(|\sym \nabla_{h(n(k))} \vartheta_k|^2)_{k \in \N}$
with (Q1) as well as (\ref{eq:22}), (\ref{eq:elias1}), (\ref{eq:elias2}), (\ref{konacno1}), (\ref{elias3})
and Proposition~\ref{identi} we get
\begin{eqnarray*}
  \liminf\limits_{n\to \infty}I^{{h_n}}(\ya^{h_n})&=& \liminf\limits_{k \to \infty }I^{h_{n(k)}}(\ya^{h_{n(k)}})\\
  &\geq&\liminf\limits_{k \to
    \infty}\int_\Omega Q^{h_{n(k)}}(x, \widetilde E_k(x))\,dx\\
  &=&\liminf_{k \to \infty} \int_{\Omega}\chi_k Q^{h_{n(k)}} \left(x,\sym G^{h_{n(k)}}-\tfrac{1}{2} (K^{h_{n(k)}})^2 \right)\,dx\\
  &=&\liminf_{k \to \infty} \int_{\Omega} \chi_k Q^{h_{n(k)}}\left(x,\iota(M_1+x_3M_2) +\nabla_{h_{n(k)}} \vartheta_k   \right)\,dx\\
  &=& \liminf_{k \to \infty} \int_{\Omega}  Q^{h_{n(k)}}\left(x,\iota(M_1+x_3M_2) +\nabla_{h_{n(k)}} \vartheta_k   \right)\,dx\\
  &\geq& \int_{\omega} Q(x',M_1,M_2)\, dx'=I^0(u,v).
\end{eqnarray*}
To deal with arbitrary $\omega$ Lipschitz one firstly takes $D\ll \omega$ of class $C^2$ and conclude that $u \in H^1(\omega,\R^2)$, $v \in H^1(\omega) \cap H^2(D)$. In the same way as above we conclude
\begin{equation} \label{sranje1}
 \liminf\limits_{n\to \infty}I^{h_n}(\ya^{h_n}) \geq \int_{D} Q(x',M_1,M_2)\, dx'\geq \tfrac{\alpha}{12} \|\nabla'^2 v\|_{L^2(D)},
 \end{equation}
where we have used (Q'1). Since the left hand side does not depend on $D$ we conclude that $v \in H^2(\omega)$. By exhausting $\omega$ with $D \ll \omega$ of regularity $C^2$ we have the claim.
\end{proof}

\subsection{Proof of Theorem \ref{thm:up1}}
\begin{proof}
The main point here is that we want to add the relaxation field $ \nabla_{h_{n}} \vartheta_{n}$, that is given by the expression (\ref{defKKK}), to standard von K\'arm\'an type expansion.  To make the procedure formally correct we would like to have that this relaxation field is bounded in $L^{\infty}$. But we only can guarantee the equi-integrability of
$ (|\sym \nabla_{h_{n}} \vartheta_{n}|^2)_{n\in\N}$. We show, using the results from Appendix, that we can replace these fields by bounded ones and then by diagonalization procedure approach asymptotic formulae.

Without any loss of generality we can assume that $\bar{R}=I$.
First we assume that $u \in C^1(\bar{\omega},\R^2)$, $v \in C^2(\bar{\omega})$. The general claim will follow by density argument and by diagonalization, which is standard in $\Gamma$-convergence.
Denote by $M_1=\sym \nabla' u-\tfrac{1}{2} \nabla' v \otimes \nabla' v$, $M_2=-\nabla'^2 v$.
By using Lemma \ref{lem:improveglavnazamjena} we find a subsequence, still denoted by $(h_n)_{n\in \N}$
and $(\varphi_{n})_{n\in \N} \subset H^2 (\omega)$ and $(\psi_{n})_{n\in N} \subset H^1 (\Omega,\R^3)$ such that such that $\varphi_{n} \to 0$ strongly in $H^1$, $\psi_{n} \to 0$ strongly in $L^2$.
Moreover, the following is valid
\begin{enumerate}[(a)]
\item  $\left(|\nabla'^2 \varphi_{n}|^2\right)_{n\in\N}$ and $\left(| \nabla_{h_n}\psi_{n} |^2\right)_{n \in \N}$ are equi-integrable,
\item $$\limsup_{n \to \infty} \left( \| \varphi_{n} \|_{H^2}+\|\nabla_{h_n}\psi_{n}\|_{L^2}\right)\leq C \left( \beta \|M_1+x_3 M_2\|^2_{L^2}+1 \right).$$
\item For $(\vartheta_{n})_{n \in \N} \subset H^1(\Omega;\R^3)$ defined by
$$ \vartheta_{n}= \psi_{n}+\left(\begin{array}{c} 0\\ 0\\ \tfrac{\varphi_{n}}{h_n} \end{array} \right)-x_3 \left(\begin{array}{c} \partial_1 \varphi_{n} \\ \partial_2 \varphi_{n} \\ 0 \end{array} \right),$$
we have
\begin{equation} \label{defKKK}
\int_{\omega} Q(x',M_1,M_2)\,dx'=\lim_{n\to \infty} \int_{\Omega} Q^{h_{n}}\left(x,\iota(M_1+x_3 M_2)
+\nabla_{h_{n}} \vartheta_{n}\right)\, dx.
\end{equation}

\end{enumerate}

We  know that $(\vartheta_{n,1},\vartheta_{n,2},h_{n} \vartheta_{n,3}) \to 0$ strongly in $L^2$
and that
 \begin{equation} \label{defsim}
  \sym \nabla_{h_{n}} \vartheta_{n}=-x_3\iota (\nabla'^2) \varphi_{n}+\sym \nabla_{h_{n}}\psi_{n},
  \end{equation}
is  $L^2$ equi-integrable.
In the same way as in the proof of Lemma \ref{lem:improveglavnazamjena} we can suppose that $\varphi_n=\nabla' \varphi_n=0$ in a neighborhood of $\partial \omega$ and that $\psi_n=0$ in a neighborhood of $\partial \omega \times I$. We extend $\varphi_n$, $\psi_n$ by zero on $\tilde{\omega}$, where $\tilde{\omega}$ has $C^{1,1}$ boundary and $\omega \subset \tilde{\omega}$.
By using Corollary \ref{kor:ekvi00} and Corollary \ref{nemoguce2} we find for each $\lambda>0$ and $n\in \N$, $\varphi^\lambda_{n} \in H^2 \left(\omega \right)$ and $\psi_{n}^\lambda \in H^1\left(\Omega,\R^3 \right)$ such that
\begin{eqnarray}
 \label{pod6}\sup_{n \in \N} \|\varphi^{\lambda}_{n}\|_{W^{2,\infty}} & \leq & C(\tilde{\omega}) \lambda, \\
\label{pod8} \lim_{\lambda \to \infty} \sup_{n \in \N} \|\varphi^{\lambda}_{n}-\varphi_{n}\|_{H^2}&=&0, \\
\label{pod7}\sup_{\lambda>0} \limsup_{n \to \infty}\|\varphi^{\lambda}_{n}\|_{H^2} & \leq & C(\tilde{\omega})\left( \beta \|M_1+x_3 M_2\|^2_{L^2}+1 \right).
\end{eqnarray}
and
\begin{eqnarray}
\label{pod2}\sup_{n \in \N}(\|\psi^{\lambda}_{n}\|_{L^\infty}+\|\nabla_{h_n} \psi^{\lambda}_{n}\|_{L^\infty}) & \leq & C(\tilde{\omega}) \lambda, \\
\label{pod3}& &\\\nonumber \lim_{\lambda \to \infty} \sup_{n \in \N} \left(\|\psi^{\lambda}_{n}-\psi_{n}\|_{L^2}+\|\nabla_{h_n} \psi^{\lambda}_{n}-\nabla_{h_n} \psi_{n}\|_{L^2}\right)&=&0, \\ \sup_{\lambda>0} \limsup_{n \to \infty}\left(\|\psi^{\lambda}_n\|_{L^2}+\|\nabla_{h_n} \psi^{\lambda}_{n}\|_{L^2}\right) & \leq & \\ \nonumber& &\hspace{-15ex}  C(\tilde{\omega})\left( \beta \|M_1+x_3 M_2\|^2_{L^2}+1 \right).
 \end{eqnarray}
Notice that as the consequence of (\ref{pod8}) and (\ref{pod3})  we have
\begin{equation}\label{teziu0}
\lim_{\lambda \to \infty} \limsup_{n \to \infty}\left( \|\varphi_{n}^\lambda\|_{H^1}+\|\psi^\lambda_{n}\|_{L^2}\right)=0.
\end{equation}
Define
$$ \vartheta_{n}^{\lambda}= \psi^{\lambda}_{n}+\left(\begin{array}{c} 0\\ 0\\ \tfrac{\varphi^{\lambda}_{n}}{h_n} \end{array} \right)-x_3 \left(\begin{array}{c} \partial_1 \varphi^{\lambda}_{n} \\ \partial_2 \varphi^{\lambda}_{n} \\ 0 \end{array} \right).$$
Again we have
 \begin{equation} \label{defsim1}
  \sym \nabla_{h_{n}} \vartheta^{\lambda}_{n}=-x_3\iota( \nabla'^2 \varphi^{\lambda}_{n})+\sym \nabla_{h_{n}}\psi^{\lambda}_{n}.
  \end{equation}
 Notice that due to (\ref{pod8}), (\ref{pod3}) we have
 \begin{equation}
\label{razlikasim} \lim_{\lambda \to \infty} \sup_{n\in \N}\|\sym \nabla_{h_{n}} \vartheta^{\lambda}_{n}-\sym \nabla_{h_{n}} \vartheta_{n}\|_{L^2}=0.
 \end{equation}
Define also for every $n,\lambda$ the function $y^{\lambda}_{n}:\Omega \to \R^3$ by
\begin{eqnarray*}
 y^{\lambda}_{n} (x',x_3)&=&\left(\begin{array}{c} x' \\ h_n x_3 \end{array} \right)+\left(\begin{array}{c}h_n^2 u(x')\\ h_n \left (v(x')+\varphi^{\lambda}_{n}(x') \right) \end{array} \right)-h_n^2 x_3 \left(\begin{array}{c} -\partial_1 \left(v+\varphi^{\lambda}_{n}\right)(x') \\ -\partial_2 \left(v+ \varphi^{\lambda}_{n} \right)(x') \\ 0 \end{array}\right) \\& &+h_n^2 \psi^\lambda_{n}(x',x_3) +\tfrac{1}{2}h_n^3 x_3\left(\begin{array}{c}0 \\0\\ |\partial_1(v+\varphi^{\lambda}_{n})(x')|^2+|\partial_2(v+\varphi^{\lambda}_{n})(x')|^2 \end{array} \right).
\end{eqnarray*}
From (\ref{teziu0}) we have:
\begin{equation} \label{prviuvjet}
\lim_{\lambda\to \infty} \limsup_{n\to \infty} \left(\left\|  \tfrac{\int_I y'^{\lambda}_{n}-x'}{h_n^2}-u\right\|_{L^2}+\left\|  \tfrac{\int_I y ^{\lambda}_{n,3}}{h_n}-v\right\|_{L^2}\right)=0,
\end{equation}
where $y'^{\lambda}_{n}=(y^{\lambda}_{n,1}\, ,\, y^{\lambda}_{n,2})$.
Also we easily conclude, by the Taylor expansion, that for every $\lambda>0$
\begin{equation}\label{strain1}
\lim_{n\to \infty}\|\tfrac{\sqrt{(\nabla_{h_n} y^{\lambda}_{n})^t\nabla_{h_n} y^{\lambda}_{n}}-I}{h_n^2}-E^\lambda_{n}\|_{L^\infty}=0,
\end{equation}
where
\begin{equation}
E^\lambda_{n}=\iota\left(\sym \nabla' u-\tfrac{1}{2} \nabla'(v+ \varphi^{\lambda}_{n}) \otimes\nabla'(v+ \varphi^{\lambda}_{n})-x_3\nabla'^2 v \right)+\sym \nabla_{h_n}\vartheta^{\lambda}_{n}.
\end{equation}
From property (W1), Lemma \ref{L:linearization} and (\ref{ocjenakv1}) we conclude that for every $\lambda>0$ we have
\begin{equation} \label{energija1}
\lim\limits_{n\to \infty} \left| \frac{1}{h_n^4}\int_\Omega
      W^{h_n}(x,\nabla_{h_n} y^{\lambda}_{n})\,dx-\int_\Omega Q^{h_n}(x, E^{\lambda}_{n}(x))\,dx\right|=0.
\end{equation}
Notice also that as a consequence of  (\ref{pod8}), (\ref{pod7}), (\ref{teziu0}), (\ref{razlikasim}) and the interpolation we have
\begin{equation} \label{strain2}
\lim_{\lambda \to \infty} \limsup_{n \to \infty}\|E^\lambda_{n}-E_{n}\|_{L^2}=0,
\end{equation}
where
\begin{equation} \label{strain3}
E_{n}=\iota\left(\sym \nabla' u-\tfrac{1}{2} \nabla' v\otimes \nabla' v-x_3\nabla'^2 v \right)+\sym \nabla_{h_n}\vartheta_{n}.
\end{equation}
From (Q1), (\ref{ocjenakv1}) and (\ref{defKKK}) we have
\begin{eqnarray} \label{energija2}
& &\lim_{\lambda \to \infty}\limsup\limits_{n\to \infty} \left|\int_\Omega Q^{h_n}(x, E^{\lambda}_{n}(x))\,dx-\int_{\omega} Q(x',M_1,M_2)\, dx'\right| =0.
\end{eqnarray}
By forming the function
\begin{eqnarray*}
 g(\lambda,n)&=&\left\|  \tfrac{\int_I y'^{\lambda}_{n}-x'}{h_n^2}-u\right\|_{L^2(\omega)}+\left\|  \tfrac{\int_I y ^{\lambda}_{n,3}}{h_n}-v\right\|_{L^2(\omega)}\\ & &+ \left| \frac{1}{h^4}\int_\Omega
      W^{h_n}(x,\nabla_{h_n} y^{\lambda}_{n})\,dx-\int_\omega Q(x', M_1,M_2)\,dx'\right|,
\end{eqnarray*}
we conclude from (\ref{prviuvjet}), (\ref{energija1}) and (\ref{energija2}) that
$$  \lim_{\lambda \to \infty} \limsup_{n \to \infty} g(\lambda,n)=0.$$
By performing diagonalizing argument we find monotone function $\lambda(n)$, such that
$\lim_{n \to \infty}$ $g(\lambda(n),n)=0$. This gives the desired sequence. To deal with $u \in H^1(\omega,\R^2)$, $v \in H^2(\omega)$ we need to do the further diagonalization. Namely, first we choose
$u_k \in C^1(\bar{\omega},\R^2)$, $v_k \in C^2(\bar{\omega})$ such that
$$ \lim_{k\to\infty}\|u_k-u\|_{H^1}=0,\ \lim_{k \to \infty} \|v_k-v\|_{H^2}=0.$$
Denote by
$$M_{1,k}=\sym \nabla'u_k-\tfrac{1}{2} \nabla' v_k \otimes \nabla'v_k,\quad M_{2,k}=-\nabla'^2 v_k.$$
We have that $M_{1,k} \to M_1$, $M_{2,k} \to M_2$ strongly in $L^2$.
Then for each $k \in \N$ we choose a sequence  of functions $(y_{k,n})_{n\in\N}\subset H^1(\Omega,\R^3)$ such that
$$\lim_{n\to \infty} \left(\left\|  \tfrac{\int_I y'_{k,n}-x'}{h_n^2}-u_k\right\|_{L^2}+\left\|  \tfrac{\int_I y_{k,n,3}}{h_n}-v_k\right\|_{L^2}\right)=0,$$
and
$$
\lim\limits_{n\to \infty} \left| \frac{1}{h_n^4}\int_\Omega
      W^{h_n}(x,\nabla_{h_n} y_{k,n})\,dx-\int_\omega Q(x', M_{1,k},M_{2,k})\,dx'\right|=0.
$$
From (Q'1) we see that
$$
\lim\limits_{k\to \infty} \left|\int_\omega Q(x', M_{1,k},M_{2,k})\,dx'-\int_\omega Q(x', M_{1},M_{2})\,dx'\right|=0.
$$
Thus for the function formed by
\begin{eqnarray*}
 g(k,n)&=& \left\|  \tfrac{\int_I y'_{k,n}-x'}{h_n^2}-u\right\|_{L^2(\omega)}+\left\|  \tfrac{\int_I y_{k,n,3}}{h_n}-v\right\|_{L^2(\omega)} \\ & &+
 \left|\tfrac{1}{h_n^4}\int_\Omega
      W^{h_n}(x,\nabla_{h_n} y_{k,n})\,dx-\int_\omega Q(x', M_{1},M_{2})\,dx'\right|,
\end{eqnarray*}
we see that $\lim_{k \to \infty} \lim_{n \to \infty} g(k,n)=0$. Then, by diagonalizing, we obtain the sequence $k(n)$ such that $\lim_{n\to \infty} g(k(n),n)=0$.
\end{proof}

\appendix

\section{Auxiliary results}
\begin{proposition} \label{propa1}                                                                                                                                                                                                                                                                                   Let $1 \leq p \leq \infty$, $\lambda>0$. Let $A$ be a bounded open set in $\R^n$ with Lipschitz boundary
\begin{enumerate}
\item Suppose $u \in W^{1,p} (A)$
Then there exists $u^\lambda  \in W^{1,\infty}(A)$ such that
\begin{eqnarray*}
\|u^\lambda\|_{W^{1,\infty}} & \leq & C(n,p,A) \lambda \\
|\{x \in A: u^\lambda(x) \neq u(x) \}| &\leq& \frac{C(n,p,A)}{\lambda^p} \int_{\{ |u|+|\nabla u| \geq
\lambda/C(n,p,A) \}} \big(|u|+|\nabla u|\big)^p \ud x. \end{eqnarray*}
In particular,
$$ \lim_{\lambda \to \infty} \Big( \lambda^p \left|\{ x \in A:\ u^\lambda(x) \neq u(x) \}\right|\Big)=0. $$
If we define Hardy Littlewood maximal function
$$   Ma(x)=\sup_{r>0}\fint_{B(x,r)} a(y) dy, $$
where $a=|\tilde{u}|+|\nabla \tilde{u}|$   ($\tilde{u}$ is the extension of $u$ to $W^{2,2}(\R^n)$ which has
the compact support)
and
$$A^\lambda=\{x \in \R^n: \ Ma(x) < \lambda \textrm{ and $x$ is a Lebesgue point of} \ u, \ \nabla u \textrm{
and }  \nabla^2 u \}, $$
then we can construct $u^\lambda$ such that
$$\{u^\lambda \neq u\}=\tilde{A}^\lambda,$$
where $\tilde{A}^\lambda$ is a closed subset of $A^\lambda \cap A$ which satisfies
$|A \backslash \tilde{A}^\lambda| \leq C|A \backslash A^\lambda|$, for some $C>1$.
\item Assume additionally  that $A$ is has the boundary of class $C^{1,1}$ and $u \in W^{2,p} (A)$.
Then there exists $u^\lambda \in W^{2,\infty} (A)$  such that
\begin{eqnarray*}
\|u^\lambda\|_{W^{2,\infty}} & \leq & C(n,p,A) \lambda, \\
|\{x \in A: u^\lambda(x) \neq u(x) \}| &\leq& \\ & & \hspace{-20ex} \frac{C(n,p,A)}{\lambda^p} \int_{\{
(|u|+|\nabla u|+|\nabla^2 u|) \geq \lambda/C(n,p,A) \}} (|u|+|\nabla u|+|\nabla^2 u|)^p \ud x,
\end{eqnarray*}
where $a=|u|+|\nabla u|+|\nabla^2 u|$.
If we define
$$   Ma(x)=\sup_{r>0}\fint_{B(x,r)} a(y) dy, $$
and
$$A^\lambda=\{x \in A: \ Ma(x) < \lambda \textrm{ and $x$ is a Lebesgue point of} \ u, \ \nabla u \textrm{
and } \  \nabla^2 u   \}, $$
then we can construct $u^\lambda$ such that
$$\{u^\lambda \neq u\}=\tilde{A}^\lambda,$$
where $\tilde{A}^\lambda$ is a closed subset of $A^\lambda$ which satisfies
$|A \backslash \tilde{A}^\lambda| \leq C|A \backslash A^\lambda|$, for some $C>1$.
\end{enumerate}
\end{proposition}                                                                                                                                                                                                                                                                                                    \begin{proof}
See the proof of Proposition A2 in \cite{FJM-02}. The condition in (b) that the domain is of class $C^{1,1}$ is  not demanded there. The argument is that one can extend $W^{2,p}(A)$ to $W^{2,p}(\R^n)$ when $A$ is only Lipschitz (see  e.g. \cite{Stein70}). However, if for $u \in W^{2,p}(A)$ we denote this extension by $Eu$ then it is not clear to the author weather the term
$$\int_{\{
(|Eu|+|\nabla Eu|+|\nabla^2 Eu|) \geq \lambda/C_1(n,p,A) \}} (|Eu|+|\nabla Eu|+|\nabla^2 Eu|)^p \ud x,$$
can be controlled with the term
$$\int_{\{
(|u|+|\nabla u|+|\nabla^2 u|) \geq \lambda/C_2(n,p,A) \}} (|u|+|\nabla u|+|\nabla^2 u|)^p \ud x.$$
For the standard extension operator, constructed using the reflexion, this can be easily proved to be valid.
\end{proof}
\begin{remark}\label{rem:jednakost}
Notice that due to \cite[Theorem 3, Section 6]{Evans92}  for $u,v \in W^{1,p}(A)$ we have that
$$\{u=v\}=\{u=v, \nabla u=\nabla v\}\cup N,$$
where $N$ is the set of measure zero. From this it follows that for $u,v \in W^{2,p}(A)$ we have that
$$\{u=v\}=\{u=v, \nabla u=\nabla v, \nabla^2 u=\nabla^2 v\}\cup N,$$
where $N$ is the set of measure zero.
\end{remark}
\begin{corollary}\label{kor:ekvi00}
Let $1 \leq p \leq \infty$, $\lambda>0$ and let $A$ be a open bounded open set in $\R^n$ with Lipschitz boundary.
\begin{enumerate}[(a)]
\item
Suppose that $(u^h)_{h>0} \subset W^{1,p}(A)$ is a sequence such that $u^h \rightharpoonup u$ weakly in $W^{1,p}$ and  $(|\nabla u^h|^p)_{h>0}$ is equi-integrable. Then there exists $(u^{\lambda,h})_{\lambda,h>0}$ such that
\begin{eqnarray*}
 \|u^{\lambda,h}\|_{W^{1,\infty}} & \leq & C(n,p,A) \lambda, \\
\lim_{\lambda \to \infty} \sup_{h>0} \|u^{\lambda,h}-u^h\|_{W^{1,p}}&=&0, \\
\|u^{\lambda,h}\|_{W^{1,p}} & \leq & C(n,p,A) \| u^h \|_{W^{1,p}}.
 \end{eqnarray*}
\item Assume additionally that $A$ has the boundary of class $C^{1,1}$ and that $(u^h)_{h>0} \subset W^{2,p}(A)$ is a sequence such that $u^h \rightharpoonup u$ weakly in $W^{2,p}$ and $(|\nabla^2 u^h|^p)_{h>0}$ is equi-integrable. Then there exists $(u^{\lambda,h})_{\lambda,h>0}$ such that
\begin{eqnarray*}
\|u^{\lambda,h}\|_{W^{2,\infty}} & \leq & C(n,p,A) \lambda, \\
\lim_{\lambda \to \infty} \sup_{h>0} \|u^{\lambda,h}-u^h\|_{W^{2,p}}&=&0, \\
\|u^{\lambda,h}\|_{W^{2,p}} & \leq & C(n,p,A) \| u^h \|_{W^{2,p}}
\end{eqnarray*}
\end{enumerate}
\end{corollary}
\begin{proof}
The proof is the direct consequence of Proposition \ref{propa1}. We will prove only (a).
For each $u^h$ and $\lambda>0$ we choose $u^{\lambda,h}$ such that
\begin{eqnarray}
\label{kor:up1}\|u^{\lambda,h}\|_{W^{1,\infty}} & \leq & C(n,p,A) \lambda \\
\label{kor:up2} |A^{\lambda,h}| &\leq& \frac{C(n,p,A)}{\lambda^p} \int_{\{ |u^h|+|\nabla u^h| \geq
\lambda/C(n,p,A) \}} \big(|u^h|+|\nabla u^h|\big)^p \ud x,
\end{eqnarray}
where $A^{\lambda,h}=\{x \in A: u^\lambda(x) \neq u(x) \}$.
Notice that since $u^h \to u$ strongly in $L^p$ and $(|\nabla u^h|^p)_{h>0}$ is equi-integrable we have that
\begin{equation}
\lim_{\lambda \to \infty} \sup_{h>0} \int_{\{ |u^h|+|\nabla u^h| \geq
\lambda/C(n,p,A) \}} \big(|u^h|+|\nabla u^h|\big)^p \ud x=0.
\end{equation}
From this we easily see that $\lim_{\lambda \to \infty} \sup_{h>0} \lambda^p |A^{\lambda,h}|=0$. Using (\ref{kor:up1}) we conclude that
\begin{eqnarray*}
\lim_{\lambda \to \infty}\sup_{h>0}\left(\| u^{h}\|_{L^p(A^{\lambda,h})}+\| \nabla u^{h}\|_{L^p(A^{\lambda,h})}\right) & \to& 0, \\
\lim_{\lambda \to \infty}\sup_{h>0}\left(\| u^{\lambda,h}\|_{L^p(A^{\lambda,h})}+\| \nabla u^{\lambda,h}\|_{L^p(A^{\lambda,h})}\right) & \to& 0.
\end{eqnarray*}
Notice also simple estimate
$$\|u^{\lambda,h}\|^p_{L^p (A^{\lambda,h}) }+\|\nabla u^{\lambda,h}\|^p_{L^p(A^{\lambda,h})} \leq  2 C(n,p,A)^2\|u^h\|^p_{W^{1,p}}        $$
From this we have the claim since
\begin{eqnarray*}
\| u^{\lambda,h}-u^h\|_{W^{1,p}} &=& \| u^{\lambda,h}-u^h\|_{L^p(A^{\lambda,h})}+\| \nabla u^{\lambda,h}-\nabla u^h\|_{L^p(A^{\lambda,h})} \\ & \leq& \| u^{\lambda,h}\|_{L^p(A^{\lambda,h})}+\| \nabla u^{\lambda,h}\|_{L^p(A^{\lambda,h})}+\|u^h\|_{L^p(A^{\lambda,h})}+\|\nabla u^h\|_{L^p(A^{\lambda,h})}. \\
\|u^{\lambda,h}\|_{W^{1,p}} &=& \|u^h\|_{L^p \left((A^{\lambda,h})^c \right)}+\|\nabla u^h\|_{L^p \left((A^{\lambda,h})^c \right)}\\ & & +\|u^{\lambda,h}\|_{L^p (A^{\lambda,h}) }+\|\nabla u^{\lambda,h}\|_{L^p(A^{\lambda,h})}.
\end{eqnarray*}
\end{proof}
The following proposition we prove by combining the ideas of extension given in \cite{BoceaFon02} and \cite{BraidesZeppieri07} with Proposition \ref{propa1}.
\begin{proposition} \label{prop:ex2}
 Let $1 \leq p \leq \infty$ and  $A$ be a bounded open set in $\R^2$ with Lipschitz boundary.
 Suppose that $u \in W^{1,p}(A \times I,\R^3)$.
Then for every $0<h<1$ there exists $u^{\lambda,h}  \in W^{1,\infty}(A \times I,\R^3)$ such that
\begin{eqnarray*}
\|u^{\lambda,h}\|_{L^\infty}+\|\nabla_h u^{\lambda,h}\|_{L^\infty} & \leq & C(n,p,A) \lambda, \\
|\{x \in A: u^{\lambda,h}(x) \neq u(x) \}| &\leq& \frac{C(n,p,A)}{\lambda^p} \int_{\{ |u|+|\nabla_h u| \geq
\lambda/C(n,p,A) \}} \big(|u|+|\nabla_h u|\big)^p \ud x. \end{eqnarray*}
\end{proposition}
\begin{proof}
The idea is to look the problem on the physical domain $A \times hI$, extend it by reflection and translation to the domain $A \times I$ and then apply  Proposition \ref{propa1} and choose the good strip.
Define $\tilde{u}:A \times I \to \R^3$ as $2h$ periodic function in the variable $x_3$ in the following way
$$ \tilde{u}^h(x',x_3)=\left\{ \begin{array}{lr}u(x',\tfrac{x_3}{h}), & \text{ if } x_3 \in hI, \\ u(x',1-\tfrac{x_3}{h}), & \text{ if } x_3 \in [h/2,3h/2], \end{array} \right.$$
and extend it by periodicity on $A \times I$.
This implies that we have $2l+1=2\lfloor\tfrac{1}{2h}-\tfrac{1}{2}\rfloor+1$ whole strips and at most $2$ strips   with the boundary $x_3=1/2$ i.e. $x_3=-1/2$ where the function $\tilde{u}^h$ does not exhaust the full period $2h$.
Denote for $i \in \{-l,\dots,l\}$ the sets
$K_i=[(2i-1)h/2,(2i+1)h/2]$ and $L_1=[(2l+1)h/2,1/2] $, $L_2=[-1/2,-(2l+1)h/2]$.
Notice that $I=\cup_{i=0}^l K_i\cup L_1\cup L_2$ and
$$\nabla \tilde{u}^h=\left\{ \begin{array}{lr} \nabla_h u(x',\tfrac{x_3}{h}), & \text{ if } x_3 \in hI, \\ \left(\partial_1 u(x',1-\tfrac{x_3}{h})\, , \,\partial_2 u(x',1-\tfrac{x_3}{h})\, ,\,-\tfrac{1}{h} \partial_3 u(x',1-\tfrac{x_3}{h}) \right)  , & \text{ if } x_3 \in [h/2,3h/2]. \end{array} \right. $$
Notice that $\tilde{u}^h \in W^{1,p}(A \times I,\R^3)$.
We apply Proposition \ref{propa1} on the function $\tilde{u}^h$ to obtain the function $\tilde{u}^{\lambda,h}$ such that
\begin{eqnarray*}
\|\tilde{u}^{\lambda,h}\|_{W^{1,\infty}} & \leq & C(n,p,A) \lambda \\
|\{x \in A: \tilde{u}^{\lambda,h}(x) \neq \tilde{u}(x) \}| &\leq& \frac{C(n,p,A)}{\lambda^p} \int_{\{ |\tilde{u}^h|+|\nabla \tilde{u}^h| \geq
\lambda/C(n,p,A) \}} \big(|\tilde{u}^h|+|\nabla \tilde{u}^h|\big)^p \ud x. \end{eqnarray*}
We want to show that there exists strip which satisfies the appropriate estimate. Notice that, due to our construction we have for every $i\in \{-l,\dots l\}$ and $j\in \{1,2\}$
\begin{eqnarray}
& &\int_{\{ |\tilde{u}^h|+|\nabla \tilde{u}^h| \geq
\lambda/C(n,p,A)  \} \cap A\times  K_i} \big(|\tilde{u}^h|+|\nabla \tilde{u}^h|\big)^p \, dx= \\ & & \hspace{15ex} \nonumber h\int_{\{ |u|+|\nabla_h u| \geq\lambda/C(n,p,A)  \}} \big(|u|+|\nabla_h u|\big)^p \, dx, \\ & &\nonumber
\int_{\{ |\tilde{u}^h|+|\nabla \tilde{u}^h| \geq
\lambda/C(n,p,A)  \} \cap A\times L_j} \big(|\tilde{u}^h|+|\nabla \tilde{u}^h|\big)^p \, dx \\ & &\hspace{15ex} \nonumber \leq h\int_{\{ |u|+|\nabla_h u| \geq\lambda/C(n,p,A)  \}} \big(|u|+|\nabla_h u|\big)^p \, dx.
\end{eqnarray}
From this we conclude that there exists strip i.e. $i\in \{-l,\dots, l\}$ and the set $A \times K_i$ such that
\begin{eqnarray*}
& &\tfrac{1}{h}|\{x \in A \times K_i: \tilde{u}^{\lambda,h}(x) \neq \tilde{u}(x) \}| \leq \\ & &\hspace{10ex} \frac{3C(n,p,A)}{\lambda^p}\int_{\{ |u|+|\nabla_h u| \geq\lambda/C(n,p,A)  \}} \big(|u|+|\nabla_h u|\big)^p \, dx.
\end{eqnarray*}
To obtain $u^{\lambda,h}$ we take $\tilde{u}^{\lambda,h}|_{A \times K_i}$, translate it to the strip $A \times [-h/2,h/2]$, if necessary reflect it, and then stretch it to the domain $A \times I$.
\end{proof}
The proof of the following corollary goes in the same way as the proof of Corollary \ref{kor:ekvi00}, using Proposition \ref{prop:ex2} instead of Proposition \ref{propa1}. We will just state the result.
\begin{corollary}\label{nemoguce2}
 Let $1 \leq p \leq \infty$ and  $A$ be a bounded open set in $\R^2$ with Lipschitz boundary.
Suppose that $(u^h)_{h>0} \subset W^{1,p}(A)$ is a sequence such that $u^h \rightharpoonup u$ weakly in $W^{1,p}$ and $(|\nabla_h u^h|^p)_{h>0}$ is equi-integrable. Then there exists $(u^{\lambda,h})_{\lambda,h>0}$ such that
\begin{eqnarray*}
\|u^{\lambda,h}\|_{L^\infty}+\|\nabla_h u^{\lambda,h}\|_{L^\infty} & \leq & C(n,p,A) \lambda, \\
\lim_{\lambda \to \infty} \sup_{h>0} \left(\|u^{\lambda,h}-u^h\|_{L^p}+\|\nabla_h u^{\lambda,h}-\nabla_h u^h\|_{L^p}\right)&=&0, \\
\|u^{\lambda,h}\|_{L^p}+\|\nabla_h u^{\lambda,h}\|_{L^p} & \leq & C(n,p,A) \left(\|u^{h}\|_{L^p}+\|\nabla_h u^{h}\|_{L^p}\right).
 \end{eqnarray*}
\end{corollary}
The following proposition is just simple adaption of \cite[Lemma 1.2]{FoMuPe98}.
\begin{proposition} \label{ekvi1}
Let $p>1$. Let $A \subset \R^n$ be a open bounded set.
\begin{enumerate}                                                                                                                                                                                                                                                                                                    \item Let $(w_n)_{n \in \N}$  be a bounded sequence in $W^{1,p}(A)$. Then there exist a subsequence $(w_{n(k)})_{k \in \N}$
    and a sequence $(z_k)_{k \in \N} \subset W^{1,p}(A)$ such that
\begin{equation}\label{stefan1} |\{z_k \neq w_{n(k)} \}| \to 0  ,
\end{equation}
as $k \to \infty$ and $\big(|\nabla z_k|^p \big)_{k \in \N}$ is equi-integrable. Each $z_k$ may be chosen to be
Lipschitz function. If $w_n \rightharpoonup w$ weakly in $W^{1,p}$ then $z_k \rightharpoonup w$ weakly in $W^{1,p}$.
\item Let $(w_n)_{n \in \N}$  be a bounded sequence in $W^{2,p}(A)$. Then there exist a subsequence $(w_{n(k)})_{k \in \N}$
    and a sequence $(z_k)_{k \in \N} \subset W^{2,p}(A)$ such that
 \begin{equation} \label{stefan2}
 |\{z_k \neq w_{n(k)} \}|                                                                                                                                                                                                        \to 0  ,
 \end{equation}                                                                                                                                                                                                                                                                                                           as $k \to \infty$ and $\big(|\nabla^2 z_k|^p \big)_{k \in \N}$ is equi-integrable. Each  $z_k$ may be chosen such
that $z_k \in W^{2,\infty}(S)$. If $w_n \rightharpoonup w$ weakly in $W^{2,p}$ then $z_k \rightharpoonup w$ weakly in $W^{2,p}$.
\end{enumerate}
\end{proposition}

\begin{proof}
Proof of (i) is given in \cite[Lemma 1.2]{FoMuPe98}. The proof of (ii) goes in the same way. We can assume that the boundary of $A$ is of class $C^{1,1}$ (to deal with general open bounded set see the proof of Step 2 in \cite[Lemma 1.2]{FoMuPe98}.
Namely, we extend each $w_n$ on $\R^n$ such that the support of each $w_n$ lies in a fixed compact subset $K                                                                                                                                                                                                         \subset \R^n$. We denote this extension also by $w_n$.
Denote by $a_n=|w_n|+|\nabla w_n|+|\nabla^2 w_n|$ and by                                                                                                                                                                                                                                                             $$   Ma_n(x)=\sup_{r>0}\fint_{B(x,r)} a_n(y) dy, $$
the Hardy Littlewood maximal function. It is well known that
\begin{equation}\label{HLes} \|M(a_{n(k)}) \|_{L^p(\R^n)} \leq C(n,p) \|w_{n(k)}\|_{W^{2,p}(\R^n)} \leq C(n,p,A)
\|w_{n(k)}\|_{W^{2,p}(A)}
\end{equation}
We denote by $\mu=\{\mu_x\}_{x \in \Omega}$  the Young measures associated with the converging subsequence of                                                                                                                                                                                                        $(M(\nabla a_n))_{n \in \mathbb{N}}$. We have the following properties
\begin{enumerate}
\item $\int_\Omega \int_{\R} |s|^p d \mu_x < +\infty$.                                                                                                                                                                                                                                                               \item whenever $(f(M(a_n))_{n \in \mathbb{N}}$ converges weakly in $L^1(\Omega)$, its weak limit is given by
    $$\bar{f}(x):=\langle \mu_x,f \rangle, \textrm{ a.e.} x \in \Omega. $$                                                                                                                                                                                                                                           \end{enumerate}
For $k \in \mathbb{N}$ we consider the truncation map $T_k: \mathbb{R} \to \mathbb{R}$ given by
$$ T_k:= \left\{\begin{array}{ll} x, & |x| \leq k, \\ k\frac{x}{|x|}, & |x|>k \end{array} \right. $$                                                                                                                                                                                                                 In the same way as in proof of Lemma \cite[Lemma 1.2]{FoMuPe98} we obtain a subsequence $w_{n(k)}$ such that                                                                                                                                                                                                         $$|T_k (M(a_{n(k)}))|^p \rightharpoonup \overline{f} \textrm{ weakly in } L^1(A),$$                                                                                                                                                                                                                                  where
$$ \overline{f}(x)=\int_{\R} |s|^p d \mu_x(s).  $$
Set
$$ \tilde{R}_k:=\{ x \in \R^n: M(a_{n(k)}) < k\}.$$
Notice that for $k$ large enough, since the support of $w_{n(k)}$ lies in $K$, we have that $\tilde{R}_k \subset
K_1$ where $K_1$ is a compact subset of $\R^n$, $K_1 \supset K$.
So without the loss of generality we can assume that for each $k$ we have $\tilde{R}_k \subset K_1$.
Denote by $R_k$ the closed subset of $\tilde{R}_k \cap A$ such that
$$ |A \backslash  \tilde{R}_k|  \leq 2 |A \backslash  R_k|, $$
and
$$ | \tilde{R}_k \backslash R_k| \leq \frac{1}{k^{p+1}}. $$
By Proposition \ref{propa1} (ii) there exists $z_k \in W^{2,\infty}(A)$  such that
$$ z_k= w_{n(k)} \textrm{ a.e. on } R_k, \ \|z_k\|_{W^{2,\infty}} \leq C(n,p,A)k.   $$
We have
\begin{eqnarray*}
|\{x \in \Omega: z_k \neq w_{n(k)} \}| &\leq& |\tilde{R}_k|+\frac{1}{k^{p+1}} \\ &\leq& \frac{1}{k^p}
\|M a_{n(k)}\|^p_{L^p}+\frac{1}{k^{p+1}},
\end{eqnarray*}
and this term tends to zero as $k \to \infty$. For a.e. $x \in  R_k$ we have
$$ |\nabla^2 z_k|=|\nabla^2 w_{n(k)}| \leq \left|M(a_{n(k)})\right|=\left|T_k(M(a_{n(k)}))\right|, $$
while if $x \in A \cap \tilde{R}_k^c$ we have
$$ \left| \nabla^2 z_{k}(x)\right| \leq C(n,p,A)k \leq C(n,p,A)\left|T_k(M(a_{n(k)}))(x)\right|.  $$
For $x \in (A \cap \tilde{R}_k) \backslash R_k)$ we can only conclude $$ |\nabla^2 z_k (x)| \leq C(n,p,A) k.   $$
Since we have
$$ \int_A |\nabla^2 z_k| dx=\int_{R_k} |\nabla^2 z_k| dx+\int_{A \backslash R_k} |\nabla^2 z_k| dx
+\int_{(A \cap \tilde{R}_k)\backslash R_k} |\nabla^2 z_k| dx. $$
and
\begin{eqnarray*}
 \int_{R_k} |\nabla^2 z_k|^p dx & \leq & \int_{R_k} |T_k(M(a_{n(k)}))|^p  ,\\
 \int_{A \backslash R_k} |\nabla^2 z_k|^p & \leq & \int_{A \backslash R_k} |T_k(M(a_{n(k)}))|^p, \\
 \int_{(A \cap \tilde{R}_k)\backslash R_k} |\nabla^2 z_k| dx &\leq&  \frac{C(n,p,A)}{k},
\end{eqnarray*}
taking into account that $T_k(M(a_{n(k)}))$ is equi-integrable, we have the claim. It is easy to see that from
the property (\ref{stefan2}) it follows  that $(z_k)_{k \in \N}$ has the same weak limit as $(w_k)_{k \in \N}$.
\end{proof}

The following proposition can be found in \cite{BoceaFon02} (see also \cite{BraidesZeppieri07}).
\begin{proposition}\label{ekvi2}
 Let $1<p<+\infty$ and $A \subset \R^2$ be a open bounded set with Lipschitz boundary. Let $(h_n)_{n\in \N}$ be a sequence of positive numbers converging to zero and let $(w_n)_{n \in\N}$ be a bounded sequence in $W^{1,p}(A \times I,\R^3)$ satisfying:
$$ \limsup_{n \in \N} \int_{A \times I} \left| \left( \partial_1 w_{n} \ \partial_2 w_n \ \tfrac{1}{h_n} \partial_3 w_n   \right)\right|^p \, d x<+\infty.$$
Suppose further that $w_n \rightharpoonup \psi$ weakly in $W^{1,p}(A \times I,\R^3)$. Then there exists a subsequence
$(w_{n(k)})_{k \in \N}$ and a sequence $(z_k)_{k \in \N}$ such that
\begin{enumerate}[(a)]
\item $\lim_{k \to \infty} |x \in A \times I: \ z_k(x) \neq w_{n(k)}(x) | =0$,
\item $\left \{\left( \partial_1 z_k \ \partial_2 z_k \ \tfrac{1}{h_k} \partial_3 z_k   \right)\right\}$ is equi-integrable,
\item $z_k \rightharpoonup \psi$  weakly in  $W^{1,p}(A \times I,\R^3)$
\end{enumerate}
\end{proposition}

The following proposition is the simple case of \cite[Proposition 11.9.]{DM93}
\begin{proposition} \label{kvforme}                                                                                                                                                                                                                                                                                                  Let $X$ be a finite dimensional vector space over the real numbers and $F:X \to [0,+\infty)$  an arbitrary function.
If
\begin{enumerate}[a)]
\item $F(0)=0$,
\item $F(tx) \leq t^2 F(x)$ for every $x \in X$ and for every $t>0$,
\item $F(x+y)+F(x-y) \leq 2 F(x)+2F(y)$ for every $x,y \in X$,                                                                                                                                                                                                                                                       \end{enumerate}
then $F$ is a quadratic form. Conversely if $F$ is a quadratic form then (a), (b), (c) are satisfied, and, in                                                                                                                                                                                                        addition,
\begin{enumerate}[a)]                                                                                                                                                                                                                                                                                                \setcounter{enumi}{3}
\item $F(tx)=t^2 F(x)$ for every $x \in X$ and for every $t \in \R$ with $t \neq 0$,
\item $F(x+y)+F(x-y)=2F(x)+2F(y)$, for every $x,y \in X$.
\end{enumerate}
\end{proposition}
If $\omega$ is a Lipschitz domain by $\mathcal{A}=\mathcal{A}(\omega)$ we denote the class of all open subsets of $\omega$; while by                                                                                                                                                                                                          $\mathcal{B}=\mathcal{B}(\omega)$ we denote the class of all Borel subsets of $\omega$.  By $\mathcal{A}_0$ we denote the class of all open sets of $\omega$ that are compactly contained in $\omega$.
The following definitions, Lemma and theorem can be found in \cite[Chapter 14]{DM93}.
\begin{definition}
For a function $\alpha:\mathcal{A} \to \overline{\R}$ we say that it is increasing if $\alpha(A) \leq \alpha(B)$, whenever $A,B \in \mathcal{A}$, $A \subset B$. We say that the increasing function
$\alpha: \mathcal{A} \to \overline{\R}$ is inner regular if
$$ \alpha(A)=\sup\{\alpha(B): B \in \mathcal{A},\ B \ll A \}.  $$
\end{definition}
\begin{definition}\label{defap00}
We say that a subset $\mathcal{D}$ of $\mathcal{A}$ is dense in $\mathcal{A}$ if for every $A,B \in \mathcal{A}$, with $A \ll B$, there exists $D \in \mathcal{D}$, such that $A \ll D \ll B$.
\end{definition}
\begin{remark}
If $\alpha:\mathcal{A} \to\overline{\R}$ is an increasing function and $\mathcal{D}$ is the dense subset
of $\mathcal{A}$ then we have that
$$\alpha(A)=\sup\{\alpha(D): D \in \mathcal{D}, \ D \ll A \}. $$
\end{remark}
\begin{definition}
Let $\alpha:\mathcal{A} \to \overline{\R}$ be non-negative increasing function. We say that
\begin{enumerate}[a)]
\item $\alpha$ is subadditive on $\mathcal{A}$ if $\alpha(A) \leq \alpha(A_1)+\alpha(A_2)$ for every $A,A_1,A_2 \in \mathcal{A}$ with $A \subset A_1 \cup A_2$;
\item $\alpha$ is superadditive on $\mathcal{A}$ if $\alpha(A) \geq \alpha(A_1)+\alpha(A_2)$ for every
$A,A_1,A_2 \in \mathcal{A}$ with $A_1 \cup A_2 \subset A$ and $A_1 \cap A_2 =\emptyset$;
\item $\alpha$ is a measure on $\mathcal{A}$ if there exists a Borel measure $\mu:\mathcal{B} \to [0,+\infty]$
such that $\alpha(A)= \mu(A)$ for every $A \in \mathcal{A}$.
\end{enumerate}
\end{definition}
The following is \cite[Lemma 14.20]{DM93}
\begin{lemma} \label{DMlemma}
Let $A,B,C \in \mathcal{A}$ with $C \ll A \cup B$. Then there exist $A', B' \in \mathcal{A}_0$ such that
$C \ll A' \cup B'$, $A' \ll A$, $B' \ll B$.
\end{lemma}
The following is Theorem \cite[Theorem 14.23]{DM93}.
\begin{theorem} \label{DMtm}
Let $\alpha:\mathcal{A} \to [0,+\infty]$ be a non-negative increasing function such that $\alpha(\emptyset)=0$. The following conditions are equivalent.
\begin{enumerate}[(i)]
\item $\alpha$ is a measure on $\mathcal{A}$; \item $\alpha$ is subadditive, superadditive and inner regular on $\mathcal{A}$.
\end{enumerate}
\end{theorem}
\begin{remark}
It can be seen that the measure $\mu$ which extends $\alpha$ is given by
$$ \mu(E)=\inf\{ \alpha(A): A \in \mathcal{A}, \ E \subset A \}. $$
\end{remark}

We give a simple lemma about the sets with Lipschitz boundary.
\begin{lemma}\label{zzzadnje}
Let $A \subset \R^n$ be an open, bounded set with Lipschitz boundary. Then $A$ has finite number of connected components.
\end{lemma}
\begin{proof}
Denote by $\{\Gamma_\alpha\}_{\alpha \in \Lambda}$ the connected components of $\partial A$.
We want to prove that there is only finitely many such components.
Suppose that there is infinite many such components. Then for each $n \in \N$ we can find $x_n \in \Gamma_n$, where $\Gamma_i \neq \Gamma_j$, for all $i \neq j$. Since $\partial A$ is compact we have that at least on a subsequence, still denoted by $(x_n)_{n \in \N}$ $x_n \to x \in \partial A$. Since $A$ has Lipschitz boundary we can find a Lipschitz frame around point $x \in \partial A$, with radius $\eps>0$. This means that there exists a
 bijective map $f_x: B(x,\eps) \to B(0,1)$ such that $f_x$ and $f_x^{-1}$ are Lipschitz continuous and such that
$f_x \left(\partial A \cap B(x,\eps)\right)=B(0,1) \cap \{x_n=0\}$ and $A\cap B(x,\eps)=B(0,1) \cap \{x_n>0\}$.  This contradicts the fact that $x_n \to x$ and that $x_n$ belong to different connected components of $\partial A$. Thus we have proved that $\partial A$ has finitely many connected components. Take now all the connected components of the set $A$ and denote them by $\{A_\alpha\}_{\alpha \in \Lambda}$.
Using that $A$ has Lipschitz boundary it is easy to see that $\partial A \subset \cup_{\alpha \in \Lambda}\partial A_{\alpha}$. Then it is easy to check that $\partial A = \cup_{\alpha \in \Lambda}\partial A_{\alpha}$. Moreover it is easy to see that for $\alpha \neq  \beta$, $\partial A_\alpha \cap \partial A_\beta =\emptyset$. Namely, if there is $x  \in \partial A_\alpha \cap \partial A_\beta$, then by taking Lipschitz frame around $x$, it can be seen that $A_\alpha$ and $A_\beta$ would be connected. Also it is easy to see that every connected component of the boundary can be part of the boundary at most one of the connected component of $A$. This implies that there can be only finitely many connected components of $A$ and that they have disjoint closure.
\end{proof}

{\bf Acknowledgement.}
The work on this paper was supported by Deutsche Forschungsgemeinschaft grant no. HO-4697/1-1 and also partially by Croatian Science Foundation grant number 9477.
\\
\bibliographystyle{alpha}

\end{document}